\documentclass[xcolor=dvipsnames,svgnames,table,reqno]{amsart}

\makeatletter
\def\reflb#1#2{\begingroup
    #2%
    \def\@currentlabel{#2}%
    \phantomsection\label{#1}\endgroup
}
\makeatother

\numberwithin{equation}{section}

\usepackage[alphabetic,nobysame]{amsrefs}
\usepackage{bm} \usepackage{amssymb} \usepackage{graphicx}
\usepackage{amscd} \usepackage{enumerate} \usepackage{cancel}
\usepackage[font=footnotesize]{caption} \usepackage{dsfont}
\usepackage[colorinlistoftodos]{todonotes} \usepackage{tikz-cd}
\usepackage{mathtools}
\usepackage[colorlinks,linkcolor={darkblue}, filecolor={darkgreen},
urlcolor={black}, citecolor={darkgreen}]{hyperref}

\definecolor{darkred}{rgb}{1,0,0} 
\definecolor{darkgreen}{rgb}{0,0.6,0}
\definecolor{darkblue}{rgb}{0,0,0.8}

\newtheoremstyle{personal}%
{12pt}%      Space above
{12pt}%      Space below
{\itshape}%         Body font
{}%         Indent amount
{\bfseries}% Theorem head font
{.}%        Punctuation after theorem head
{.5em}%     Space after theorem head
{}%         Theorem head spec (can be left empty, meaning "normal")
\theoremstyle{personal}%
\newtheorem{thm}{Theorem}[section]

\newtheorem{lem}[thm]{Lemma}
\newtheorem{prop}[thm]{Proposition}

\newtheorem{maintheorem}{Theorem}
\newtheorem{maincor}[maintheorem]{Corollary}

\theoremstyle{definition}
\newtheorem{Definition}[thm]{Definition} 
\newtheorem{rem}[thm]{Remark}

\newcommand{\vol}{\mathrm{vol}}
\newcommand{\hhbar}{\hbar} 
\newcommand{\hhtop}{h_{\mathrm{top}}}
\newcommand{\hhvol}{h_{\mathrm{vol}}} 
\newcommand{\interior}{\mathrm{int}}
\newcommand{\W}{W^{1,2}}
\newcommand{\Wloc}{W^{1,2}_{\mathrm{loc}}}
\newcommand{\N}{\mathds{N}}
\newcommand{\Z}{\mathds{Z}}
\newcommand{\R}{\mathds{R}}

\newcommand{\E}{\mathds{E}}
\newcommand{\F}{\mathds{F}}
\newcommand{\FF}{\mathcal{F}}
\newcommand{\PP}{\mathcal{P}}

\newcommand{\B}{\mathcal{B}}
\newcommand{\CC}{\mathcal{C}}
\newcommand{\NN}{\mathcal{N}}
\newcommand{\XX}{\mathcal{X}}
\newcommand{\YY}{\mathcal{Y}}
\newcommand{\WW}{\mathcal{W}}
\newcommand{\UU}{\mathcal{U}}

\newcommand{\VV}{\mathcal{V}}

\newcommand{\GG}{\mathcal{G}}
\newcommand{\EE}{\mathcal{E}}
\newcommand{\zerosection}{0\mbox{-}\mathrm{section}}
\newcommand{\crit}{\mathrm{crit}}

\newcommand{\id}{\mathrm{id}}
\newcommand{\im}{\mathrm{im}}
\newcommand{\xx}{\bm{x}}
\newcommand{\yy}{\bm{y}}
\newcommand{\vv}{\bm{v}}
\newcommand{\ww}{\bm{w}}

\newcommand{\eps}{\epsilon}

\DeclareMathOperator{\supp}{\mathrm{supp}} 
\DeclareMathOperator*{\injrad}{\mathrm{injrad}}

\renewcommand{\twoheadrightarrow}{\mathrel{\mathrlap{\rightarrow}\mkern-2.3mu\rightarrow}}

\DeclareRobustCommand{\llongrightarrow}{\relbar\joinrel\relbar\joinrel\rightarrow}

\DeclareRobustCommand{\llongepi}{\relbar\joinrel\relbar\joinrel\twoheadrightarrow}
\DeclareMathOperator*{\toup}{\longrightarrow}
\DeclareMathOperator*{\ttoup}{\llongrightarrow}

\DeclareMathOperator*{\bijection}{\longleftrightarrow}

\DeclareMathOperator*{\eepi}{\llongepi}

\definecolor{viola}{HTML}{9F00FF}

\usepackage[colorinlistoftodos]{todonotes}

\begin{document}

\title{Barcode entropy of geodesic flows}

\author[V.\ L.\ Ginzburg]{Viktor L.\ Ginzburg} \address{Viktor
  Ginzburg\newline\indent Department of Mathematics, UC Santa Cruz,
  Santa Cruz, CA 95064, USA} \email{ginzburg@ucsc.edu}

\author[B.\ Z.\ G\"urel]{Ba\c{s}ak Z.\ G\"urel} \address{Ba\c{s}ak
  G\"urel\newline\indent Department of Mathematics, UCF, Orlando, FL
  32816, USA} \email{basak.gurel@ucf.edu}

\author[M. Mazzucchelli]{Marco Mazzucchelli} \address{Marco
  Mazzucchelli\newline\indent CNRS, UMPA, \'Ecole Normale Sup\'erieure
  de Lyon, 69364 Lyon, France} \email{marco.mazzucchelli@ens-lyon.fr}

\date{December 1, 2022}

\begin{abstract}

  We introduce and study the barcode entropy for geodesic flows of
  closed Riemannian manifolds, which measures the exponential growth
  rate of the number of not-too-short bars in the Morse-theoretic
  barcode of the energy functional. We prove that the barcode entropy
  bounds from below the topological entropy of the geodesic flow and,
  conversely, bounds from above the topological entropy of any
  hyperbolic compact invariant set.  As a consequence, for Riemannian
  metrics on surfaces, the barcode entropy is equal to the topological
  entropy. A key to the proofs and of independent interest is a
  crossing energy theorem for gradient flow lines of the energy
  functional.
\end{abstract}



\thanks{This work is partially supported by NSF CAREER award
  DMS-1454342 (BG), Simons Foundation Collaboration Grants 581382 (VG)
  and 855299 (BG) and the ANR grants CoSyDy, ANR-CE40-0014, and COSY,
  ANR-21-CE40-0002 (MM)}

\maketitle

\vspace{-.8cm} 
\tableofcontents

\section{Introduction}

We introduce and study the \emph{barcode entropy} of geodesic flows of
closed Riemannian manifolds, an entropy-type invariant based on the
Morse theory of the energy functional on the free loop
space. Intuitively, the barcode entropy measures the exponential
growth rate of the ``noise'' in the Morse-theoretic persistence module
associated to the energy functional. Our main results show that
barcode entropy is closely related to topological entropy and that, in
dimension two, the two invariants of the metric are actually
equal. This work can be viewed as an analogue for geodesic flows of
the constructions and results for compactly supported Hamiltonian
diffeomorphisms from \cite{Cineli:2021aa} by \c{C}ineli and the first
two authors. Nevertheless, our arguments are technically very
different, and the two papers are independent of each other. In fact,
barcode entropy is a very general concept which can be adopted in a
variety of frameworks, and the setting of geodesic flows is at least
as natural and perhaps even more natural than that of Hamiltonian
diffeomorphisms.

Let us now give an overview of our results and discuss the broader
context. Consider the energy functional
$\EE\colon\Lambda\to[0,\infty)$ on the free loop space $\Lambda$ of a
closed Riemannian manifold $(M,g)$, whose non-trivial critical points
are the closed geodesics.  The homology of the sublevel sets
$\{\EE<c^2\}\subset\Lambda$ with coefficients in some field $\F$
naturally forms a persistence module, which has an associated
barcode. We denote by $b_{\eps, c}$ the number of bars of size at
least $\eps>0$ and intersecting the interval $[0,c]$. The
$\eps$-barcode entropy $\hhbar_\eps$ is the exponential growth rate of
the function $c\mapsto b_{\eps, c}$, and the barcode entropy
$\hbar=\hbar(g;\F)$ is obtained by taking the limit as $\eps\to 0^+$;
see Definition \ref{def:hhbar}.

Two other notions of entropy associated with the geodesic flow
$\phi_t\colon SM\to SM$ of $(M,g)$ are its topological entropy
$\hhtop=\hhtop(g)$, i.e., the topological entropy of the time-one map
$\phi_1$, and the volume-growth entropy $\hhvol=\hhvol(g)$ defined
here as the exponential growth rate of the function
$t\mapsto \vol(\mathrm{graph}(\phi_t))$. (Note that this is different
from other more standard notions of volume entropy.) The celebrated
Yomdin theorem, \cite{Yomdin:1987aa}, implies that $\hhvol\leq\hhtop$.
Our first main result, Theorem~\ref{mt:bar<top_plus}, asserts that
\[\hhbar\leq\hhvol.\]
In particular, since the topological entropy of any (smooth) geodesic flow is always finite, we infer that the barcode entropy $\hhbar$ is finite as well.

The second main result of the paper, Theorem~\ref{mt:bar>top_I},
guarantees that the barcode entropy is a non-trivial invariant.
Namely, $\hhbar\geq \hhtop(I)$ for any hyperbolic compact invariant
subset $I\subset SM$, where $\hhtop(I)$ denotes the topological
entropy of the restricted geodesic flow $\phi_t|_I$. As a consequence,
when $\dim(M)=2$,
\begin{align*}
\hhbar=\hhvol=\hhtop.
\end{align*}
This is Corollary~\ref{c:surfaces} which follows from
Theorems~\ref{mt:bar<top_plus} and~\ref{mt:bar>top_I} and the results
from \cite{Katok:1980aa, Lian:2012aa, Lima:2019aa} asserting that the
topological entropy of any flow on a 3-dimensional closed manifold can
be approximated by the topological entropy of a suitable hyperbolic
compact invariant subset.  Surprisingly, the identity $\hhvol=\hhtop$
for Riemannian closed surfaces appears to be new, although it is
reminiscent of an identity due to Ma\~n\'e, \cite{Mane:1997aa}.

The central new feature of our approach distinguishing it from many
other results on topological entropy of geodesic flows via closed
geodesics (see, e.g., \cite{Dinaburg:1971aa, Katok:1982aa,
Paternain:1999aa} and references therein) is that
the topology of $M$ plays no role here.  The barcode entropy can be
positive regardless of the growth of $\pi_1(M)$ or $H_*(\Lambda)$,
e.g., even when $M$ is the sphere or a torus. For instance, by Theorem
\ref{mt:bar>top_I}, this is the case when $\phi_t$ has a localized
hyperbolic set with positive entropy. However, exponential growth of
the fundamental group and, in many instances, of the homology implies
exponential growth of the function $c\mapsto b_{\eps,c}$ (for any
$\eps>0$) and thus Theorem \ref{mt:bar<top_plus} generalizes some of
these results. Overall, without Theorems \ref{mt:bar<top_plus} and
\ref{mt:bar>top_I}, it is \emph{a priori} unclear if the Morse complex
of $\EE$ carries enough information, and if it does how to extract it,
to detect positive topological entropy in the absence of homological
or the fundamental group growth.

Likewise, our approach is different from modern generalizations in,
e.g., \cite{Abbondandolo:2021aa, Alves:2016aa, Alves:2019aa,
  Alves-etal:2019aa, Alves-Meiwes:2019aa, Alves:2022aa, Alves:2020aa,
  Macarini:2011aa}, of these classical results, relating topological
entropy of Reeb flows to their Floer theoretic invariants (e.g.,
symplectic or contact homology). Again, the contact topology of the
underlying manifold -- $SM$ in our case -- is central to those results
but essentially immaterial to Theorems~\ref{mt:bar<top_plus}
and~\ref{mt:bar>top_I}. There is however some overlap between methods used, e.g., in \cite{Macarini:2011aa, Meiwes:2018aa} and in the present paper.

One should also keep in mind that in the definition of barcode entropy
one cannot replace the bar count $b_{\eps,c}$ by the number of closed
geodesics with energy less than $c^2$, i.e., the number of critical
circles of the energy functional $\EE$ in the sublevel set
$\{\EE<c^2\}$. Indeed, the number of closed geodesics of energy at
most $c^2$ may grow arbitrarily fast in $c$, possibly
super-exponentially, even for bumpy metrics (see \cite{Burns:1996aa},
and also \cite{Asaoka:2017aa, Kaloshin:1999aa} for results on a
similar phenomenon in the Hamiltonian setting), whereas the barcode
entropy is always finite. On the technical level, the key point is
that $b_{\eps,c}$ is stable under small perturbations of the energy
functional $\EE$, while the count of critical circles is not.

The proof of Theorem \ref{mt:bar<top_plus} employs an analogue for
geodesic flows of Lagrangian tomographs from \cite{Cineli:2021aa,
  Cineli:2022aa}, a tool from integral geometry reminiscent of the
constructions in \cite{Alvarez:1998aa, Gelfand:1994aa,
  Guillemin:2005aa}. We prove Theorem \ref{mt:bar>top_I} by an
argument of different nature, based on a result of independent
interest. This is Theorem \ref{mt:bars_lower_bound} which, in
particular, implies that every closed geodesic in a locally maximal
hyperbolic compact invariant set $I\subset SM$ gives rise to a bar of
length at least $\delta$, for some constant $\delta>0$ depending only
on $I$.  In turn, the proof of Theorem \ref{mt:bars_lower_bound}
ultimately relies on a uniform crossing energy bound,
Proposition~\ref{p:crossing_energy_period_1}, which is a
Morse-theoretic counterpart of the crossing energy bound for
Hamiltonian diffeomorphisms and Floer cylinders from
\cite{Ginzburg:2014aa, Ginzburg:2018aa}; see also \cite{Allais:2020aa}
where the lower bound is proved using generating functions. However,
Proposition \ref{p:crossing_energy_period_1} is not a consequence of
these results; its proof is technically and conceptually quite
different and hinges on finite-dimensional approximations.

It is worth emphasizing that in spite of obvious parallels, Theorems
\ref{mt:bar<top_plus} and \ref{mt:bar>top_I} do not follow from their
counterparts for compactly supported Hamiltonian diffeomorphisms,
\cite[Thms.\ A and B]{Cineli:2021aa}. While there are ways to study
geodesic flows via Floer theory by encapsulating the geodesic flow
$\phi_t$ into an appropriate Hamiltonian diffeomorphism of $T^*M$ (see
\cite{Abbondandolo:2006aa, Salamon:2006aa, Viterbo:1999aa}), this
approach does not seem suitable for establishing the lower bound
on the barcode entropy as in Theorem~\ref{mt:bar>top_I}. The essential
reason is that a hyperbolic compact invariant subset of the geodesic
flow corresponds to an invariant subset of any associated Hamiltonian
diffeomorphism that is only partially hyperbolic.

One may expect the results from this paper to extend to more general
Reeb flows when a suitable analogue of Morse theory (e.g., symplectic
homology or cylindrical contact homology) is available. Then two
problems arise. The first one is defining barcode entropy and
establishing a variant of Theorem \ref{mt:bar<top_plus}; after our paper has appeared, this has been done for Reeb flows on the boundary of certain Liouville domains \cite{Fender:2023aa}. The second problem
is proving an analogue of Theorem~\ref{mt:bar>top_I} for general Reeb
flows. The difficulty here is that it is absolutely unclear at the
moment how to prove a version of the crossing energy bound in such a
general setting.

Finally, while the exponential growth rate is relevant to the
relation between barcode entropy and topological entropy, in other contexts
the noise or chaos in the system might be captured by the function
$c\mapsto b_{c,\eps}$ in different ways. For instance, one may
consider polynomial growth rate as in slow-entropy type invariants;
see, e.g., \cite[Sect.\ 2.3]{Cineli:2022ab} for applications to the
$\gamma$-norm bounds and \cite{Buhovsky:2022aa, Cohen-Steiner:2010aa}
for applications to nodal count.

\section{Main results}
\label{s:main}

\subsection{Morse-theoretic barcodes for closed geodesics}
\label{ss:barcode}

Let $(M,g)$ be a closed Riemannian manifold of dimension at least two (throughout this paper, all Riemannian metrics are tacitly assumed to be smooth, that is, $C^\infty$).
We denote by $SM$ the associated unit tangent bundle, i.e.,
$SM = \big\{ v \in TM\ \big|\ \|v\|_g=1 \big\}$.  The geodesic flow is
defined by $\phi_t\colon SM\to SM$,
$\phi_t(\dot\gamma(0))=\dot\gamma(t)$, where $\gamma\colon \R\to M$ is
a geodesic parametrized with unit speed $\|\dot\gamma\|_g=1$. The
closed geodesics $\gamma$ of length $c$ correspond to the closed
orbits $\dot\gamma$ with minimal period $c$ of the geodesic flow. An
alternative way of characterizing closed geodesics is by means of
their classical variational principle which we now recall.

We consider the free loop space $\Lambda:=\W(S^1,M)$, where $S^1=\R/\Z$,
and the energy functional $\EE\colon\Lambda\to[0,\infty)$ given by
\begin{align}
\label{eq:EE}
 \EE(\gamma)=\int_0^1 \|\dot\gamma(t)\|_g^2\, dt.
\end{align}
The circle $S^1$ acts on $\Lambda$ by time translation as
$t\cdot\gamma=\gamma(t+\cdot)\in\Lambda$, where~$t\in S^1$ and
$\gamma\in\Lambda$, and the energy functional $\EE$ is invariant under
this action. Besides the constant curves $\Lambda^0:=\EE^{-1}(0)$, the
critical point set
\[
\crit^+(\EE):=\crit(\EE)\cap\EE^{-1}(0,\infty)
\] 
consists of all critical circles $S^1\cdot\gamma$, where $\gamma$ is
any 1-periodic closed geodesic. Since closed geodesics are
parametrized proportionally to arc length, any critical value
$\EE(\gamma)$ is the squared length of the closed geodesic $\gamma$,
where $\gamma$ is viewed as a path $\gamma\colon [0,1]\to M$, i.e.
\begin{align*}
\sqrt{\EE(\gamma)}=\int_0^1 \|\dot\gamma\|_g\,dt.
\end{align*}
In this paper we shall 
work with the critical values of the functional $\sqrt{\EE}$. For this reason, for each
subset $\UU\subseteq\Lambda$ and $b>0$, we shall denote
\[\UU^{<b}:=\UU\cap\EE^{-1}[0,b^2).\] 
We emphasize that the variational principle of $\EE$ allows us to
detect all the closed geodesics of $(M,g)$, since any closed geodesic can
be reparametrized to become 1-periodic. Moreover, any closed geodesic
$\gamma$ with minimal period 1 gives rise to infinitely many critical
circles $S^1\cdot\gamma^m\subset\crit^+(\EE)$, where $m\in\N$ and
$\gamma^m:=\gamma(m\,\cdot)$ is the $m$-th iterate of $\gamma$.

We shall employ the language of persistence modules and barcodes, and
refer the reader to \cite{Polterovich:2020aa} for a comprehensive
introduction to the subject.  We consider the persistence module
$(H_b,\,i_{b,a})_{b>a>0}$, where
$H_{b}:=H_*(\Lambda^{<b},\Lambda^0;\F)$ and the maps
$i_{b,a}\colon H_a\to H_b$ are the homomorphisms induced by the
inclusion. The value $b=\infty$ is allowed, in which case we have
$H_\infty=H_*(\Lambda,\Lambda^0;\F)$. Hereafter, all singular homology
groups will be taken with coefficients in an arbitrary field $\F$,
usually suppressed in the notation.

For $c\in(0,\infty]$ and $h\in H_c$, we define the birth
and death values 
\begin{align*}
  \alpha(h)&:=\inf\big\{a < c\ \big|\ h\in
             \textrm{im}(i_{c,a})\big\}\in[0,c),\\
  \beta(h)&:=\inf\big\{b > c\ \big|\
            h\in\ker(i_{b,c})\big\}\in[c,\infty].
\end{align*}
Here, we adopt the usual convention that $\inf\varnothing=\infty$. If
$h\neq0$, we have $\alpha(h)>0$. The \emph{closed geodesics barcode}
$\B=\B(g;\F)$ is the collection of all pairs $([a,b),n)$, where
$[a,b)\subset(0,\infty)$ is an interval (of finite or infinite length)
and $n=n_{[a,b)}$ is a non-negative integer, defined as follows. For
any interval $[a,b)$ and any $c\in(a,b]$, consider the vector spaces
\begin{align*}
V & :=\big\{ h\in H_c\ \big|\ \alpha(h)\leq a,\ \beta(h)\leq b \big\},\\
W & :=\big\{ h\in H_c\ \big|\ \alpha(h)< a,\ \beta(h)\leq b \big\},\\
Z & :=\big\{ h\in H_c\ \big|\ \alpha(h)\leq a,\ \beta(h)< b \big\},
\end{align*}
and define $n_{[a,b)}=\dim V/(W+Z)$. The vector space $V$ is
finite-dimensional, and the quotient dimension $n_{[a,b)}$ is
independent of the choice of the value $c\in(a,b]$; see
Lemma~\ref{l:n_finite}. 

The barcode should be seen as a collection of real intervals
$[a,b)$, called \emph{bars}, with possible repetitions: an element
$([a,b),n)\in\B$ corresponds to $n$ copies of the bar $[a,b)$; if
$n=0$, the bar $[a,b)$ is not contained in the barcode. The size of a
bar $[a,b)$ is its length $b-a\in(0,\infty]$.  For instance, we will
say that $c>0$ is a boundary point of at least $m$ bars with size at
least $\delta$ if there exist distinct elements
$ ([a_1,c),n_1),\ldots ,([a_k,c),n_k),([c,b_1),m_1),\ldots
,([c,b_h),m_h)\in\B $ such that
$n_1+\ldots +n_k+m_1+\ldots +m_h\geq m$, and all bars $[a_i,c)$ and
$[c,b_i)$ have size $\ge \delta$.

\subsection{Barcode entropy}\label{ss:main_results}

For $\epsilon>0$ and $c>0$, we denote by
$\B_{\epsilon,c}\subset\B=\B(g;\F)$ the subcollection of those bars
of size at least $\epsilon$ that intersect the interval $(0,c]$
non-trivially, i.e.,
\[\B_{\epsilon,c}:=\big\{([a,b),n)\in\B\ \big|\ a\leq c,\
b-a\geq\epsilon \big\}.\] 
We set $b_{\epsilon,c}$ to be the number of bars
in $\B_{\epsilon,c}$, namely 
\begin{align*}
  b_{\epsilon,c}:=\!\!\!\sum_{([a,b),n)\in\B_{\epsilon,c}} \!\!\!\!\!\! n.
\end{align*}
Even though the barcode $\B$ can contain infinitely many bars
intersecting the bounded interval $(0,c]$, the value $b_{\epsilon,c}$
is always finite for any $\epsilon>0$; see
Lemma~\ref{l:b_epsilon_c_finite}.

\begin{Definition}
\label{def:hhbar}
The \emph{barcode entropy} $\hhbar=\hhbar(g;\F)$ is the limit
\[
  \hhbar:=\lim_{\epsilon\to 0^+}\hhbar_\eps,
\]
where
\[\hhbar_\eps:=\limsup_{c\to\infty} \frac{\log^+(b_{\epsilon,c})}{c}.\]
Here $\log^+:=\log(\max\{1,\cdot\})$ with the logarithm taken base
$2$.

\end{Definition}

We shall compare the barcode entropy with two other classical notions
of entropy for a closed Riemannian manifold $(M,g)$, and refer the
reader to, e.g., Paternain's monograph~\cite{Paternain:1999aa} for the
background on the subject. The first one is the topological entropy
$\hhtop=\hhtop(g)$ of the geodesic flow $\phi_t:SM\to SM$, which is
the same as the topological entropy of its time-1 map $\phi_1$. Since
$\phi_t$ is smooth, $\hhtop$ is always finite.

The second notion of entropy, which we call \emph{volume-growth
  entropy}\footnote{The volume-growth entropy should not be confused
  with another classical notion of entropy involving volumes: the
  exponential growth rate of the volume of Riemannian balls in the
  universal cover of the Riemannian manifold. The inequality of
  Theorem~\ref{mt:bar<top_plus} is reminiscent of Manning's
  inequality, \cite{Manning:1979aa}, involving this latter notion of
  volume-growth entropy, but is not directly related to it.} and
denote by $\hhvol=\hhvol(g)$, is defined by
\begin{align*}
  \hhvol:=\limsup_{t\to\infty} \frac{\log(V(t))}{t},
\end{align*}
where $V(t)$ is the volume of the graph of $\phi_t:SM\to SM$ measured
with respect to an arbitrary Riemannian metric on $SM\times SM$. The
value of $\hhvol$ is independent of the choice of this Riemannian
metric, but clearly depends on the Riemannian metric $g$ defining the
geodesic flow $\phi_t$. The celebrated Yomdin theorem
\cite{Yomdin:1987aa} implies
\[\hhvol \leq \hhtop.\] 
Indeed, the flow $\psi_t=(\id,\phi_t):SM\times SM\to SM\times SM$ also has topological entropy $\hhtop$, and the graph of $\phi_t$ is precisely $\psi_t(\Delta)$, where $\Delta=\{(v,v)\ |\ v\in SM\}$ is the diagonal submanifold.

Our first main result is the following. Together with Yomdin theorem
and the finiteness of topological entropy, it implies in particular
that the barcode entropy is always finite.

\begin{maintheorem}
\label{mt:bar<top_plus}
On any closed Riemannian manifold and for any coefficient field, we
have $\hhbar\leq\hhvol$.
\end{maintheorem}

The proof of this theorem is based on an inequality reminiscent of the classical Crofton formula applied to certain Lagrangian tomographs. For general Reeb flows of closed contact-type hypersurfaces of symplectic manifolds, a different argument due to Meiwes \cite[Prop.~10.9]{Meiwes:2018aa} and still involving Lagrangian tomographs allows to bound $\hhvol$ from below by the exponential growth-rate of certain leafwise intersections.

Knowing that the barcode entropy is always finite, it remains to
establish whether it is a non-trivial invariant, that is, whether it
does not always vanish.  Our second main result implies that the
barcode entropy is positive when the geodesic flow admits a suspended
horseshoe. For each compact subset $I\subset SM$ invariant under the
geodesic flow (i.e., $\phi_t(I)=I$ for all $t\in\R$), we denote by
$\hhtop(I)=\hhtop(I;M,g)$ the topological entropy of the restricted
geodesic flow $\phi_t|_I$. The precise statement is the following.

\begin{maintheorem}\label{mt:bar>top_I}
  On any closed Riemannian manifold and for any coefficient field, if
  $I$ is a hyperbolic compact invariant subset of the geodesic flow,
  then $\hhbar\geq\hhtop(I)$.
\end{maintheorem}

For flows on 3-dimensional closed manifolds, the topological entropy
can be approximated by the topological entropy of suitable hyperbolic
compact invariant subsets; see \cite{Katok:1980aa, Lian:2012aa,
  Lima:2019aa}.  Combining this fact with Theorems
\ref{mt:bar<top_plus} and \ref{mt:bar>top_I}, we obtain the following
corollary for Riemannian closed surfaces.

\begin{maincor}\label{c:surfaces}
  On any closed Riemannian surface and for any coefficient field, we
  have $\hhbar=\hhvol=\hhtop$.
\end{maincor}

\begin{rem}
  A priori, the barcode entropy depends on the choice of the
  coefficient field $\F$ employed in the homology group of the
  persistence module, although we do not know examples where this
  actually happens. Corollary~\ref{c:surfaces} implies in particular
  that, at least for Riemannian surfaces, the barcode entropy is
  independent of the coefficient field.
\end{rem}

\begin{rem}
  Surprisingly, the identity $\hhvol=\hhtop$ on all closed Riemannian
  surfaces provided by Corollary~\ref{c:surfaces} is new. We do not
  know whether the identity holds for higher dimensional closed
  Riemannian manifolds as well. Nevertheless, the following similar
  identity due to Ma\~n\'e, \cite{Mane:1997aa}, does hold in every
  dimension:
\begin{align*}
  \hhtop=\lim_{t\to\infty}
  \frac 1t
  \log\bigg( \underbrace{\int_M \vol(\phi_t(S_xM))\, dx}_{(*)} \bigg).
\end{align*}
Here $dx$ is the Riemannian volume form, and the volume of
$\phi_t(S_xM)$ is measured with respect to the Riemannian metric on
$SM$ induced by the one on $M$. Despite the similarities, the quantity
$(*)$ seems different from the volume $V(t)$ entering in the
definition of $\hhvol$, and Ma\~n\'e's identity does not seem to be
equivalent to ours. Both our identity and Ma\~n\'e's one rely on
Yomdin's inequality to bound $\hhtop$ from below, but Ma\~n\'e's
identity further relies on the Przytycki inequality,
\cite{Przytycki:1980aa}, to bound $\hhtop$ from above, whereas we
employ a completely different argument based on a crossing energy
bound, as we shall explain in Section~\ref{ss:invariant_subsets}.
\end{rem}

\begin{rem}
  There are other alternatives to the definition of barcode entropy
  adopted here. First, we could have instead worked with a variant of
  \emph{sequential entropy} for geodesic flows defined similarly to
  sequential barcode entropy for compactly supported Hamiltonian
  diffeomorphisms; see
  \cite{Cineli:2022aa}. Theorems~\ref{mt:bar<top_plus}
  and~\ref{mt:bar>top_I} would hold for sequential barcode entropy,
  and while we do not know if in general the two types of barcode
  entropy are equal, this would be the case when $M$ is a
  surface. Secondly, we could have used the $S^1$-equivariant Morse
  theory rather than the ordinary Morse theory. By the Gysin sequence
  and \cite[Thm.\ 3.1]{Buhovsky:2022aa}, the equivariant barcode
  entropy is greater than or equal to the barcode entropy, and hence
  Theorem \ref{mt:bar>top_I} would still hold for it. However, we do
  not know whether Theorem \ref{mt:bar<top_plus} remain true in the
  equivariant setting. Finally, the definition of barcode entropy and
  Theorem~\ref{mt:bar<top_plus} have also a relative analogue along the
  lines of relative barcode entropy from \cite{Cineli:2021aa}; we will
  touch upon it in Section~\ref{sec:relative}.
\end{rem}

\subsection{Invariant subsets and a lower bound on the bar size}
\label{ss:invariant_subsets}

The main ingredient of the proof of Theorem~\ref{mt:bar>top_I} is a
uniform lower bound on the size of the bars associated with a locally
maximal, hyperbolic, compact invariant subset of the geodesic
flow. This lower bound actually requires a slightly weaker assumption
on the invariant subset than hyperbolicity: expansivity.
Before stating the result, Theorem~\ref{mt:bars_lower_bound}, let us
introduce some notation and terminology.

Let $I\subset SM$ be a compact invariant subset for the geodesic flow
$\phi_t$, i.e., $\phi_t(I)=I$ for all $t\in\R$. We denote by
$\PP(I)\subset\crit^+(\EE)$ the space of 1-periodic closed geodesics
tangent to $I$, i.e.
\begin{align}
\label{e:P(I)}
  \PP(I):=\big\{ \gamma\in\crit^+(\EE)\ \big|\
  \tfrac{\dot\gamma(t)}{\|\dot\gamma(t)\|_g}\in I, \
  \forall t\in S^1 \big\}.
\end{align}
For each $c>0$, set $\PP^c(I):=\PP(I)\cap\EE^{-1}(c^2)$.  The
\emph{length spectrum} of $I$ is defined as
$$\sigma(I):=\{ c>0\ |\ \PP^c(I)\neq\varnothing \}.$$
In other words, $\sigma(I)$ consists of the energy of all (possibly
iterated) 1-periodic geodesics tangent to $I$. We denote the total
local homology of the set formed by closed geodesics tangent to $I$
with energy $c^2$ by
\begin{align*}
 C_*(\PP^c(I))
 :=
 H_*(\Lambda^{<c}\cup\PP^c(I),\Lambda^{<c}).
\end{align*}

We need to impose two conditions on compact invariant sets
$I\subset SM$. The first one is that $I$ is \emph{locally maximal}:
the set $I$ admits an open neighborhood $U\subset SM$, called an
\emph{isolating neighborhood}, such that
\begin{align*}
 I=\bigcap_{t\in\R} \phi_t(U).
\end{align*}
In other words, $I$ is the largest invariant subset contained in
$U$. The second condition is that $I$ is \emph{expansive}: for every
$\epsilon>0$, there exists $\delta>0$ such that for
any $z_1,z_2\in I$ and 
continuous function $s \colon \R\to\R$ 
satisfying
\[
  s(0)=0, \qquad \sup_{t\in\R}\tilde
  d\left(\phi_t(z_1),\phi_{s(t)}(z_2)\right)<\delta,
\]
we have $z_2=\phi_t(z_1)$ for some $t\in[-\epsilon,\epsilon]$. Here
$\tilde d\colon SM\times SM\to[0,\infty)$ is the distance on the unit
tangent bundle induced by the Riemannian metric $g$. This condition
was first introduced by Bowen and Walters in \cite{Bowen:1972us}.

We are now in a position to state our third main result.

\begin{maintheorem}
\label{mt:bars_lower_bound}
Let $(M,g)$ be a closed Riemannian manifold, let $I\subset SM$ be a
locally maximal, expansive, compact invariant subset of its geodesic
flow, and let $\F$ be any coefficient field. There exists $\delta>0$
such that every $c\in\sigma(I)$ is a boundary point of at least
$\dim C_*(\PP^c(I))$ bars of size at least $\delta$ in the closed
geodesics barcode $\B(g;\F)$.
\end{maintheorem}

Next, let $\gamma\in\crit^+(\EE)$ be a closed geodesic of energy
$\EE(\gamma)=c^2$. The corresponding $c$-periodic orbit of the
geodesic flow is given by $t\mapsto\phi_t(v)$, where
$v=\dot\gamma(0)/\|\dot\gamma(0)\|_g$. The closed geodesic $\gamma$ is
called \emph{non-degenerate} when $\ker(d\phi_{c}(v)-\id)$ is 1-dimensional; equivalently, the critical circle
$S^1\cdot\gamma\subset\crit^+(\EE)$ is non-degenerate in the sense of
Morse--Bott theory. An invariant subset $I\subset SM$ for the geodesic
flow $\phi_t$ is called non-degenerate when every closed geodesic
tangent to $I$ is non-degenerate. Note that this definition requires
all iterates of closed geodesics tangent to $I$ to be non-degenerate.

Furthermore, we say that a critical circle
$S^1\cdot\gamma\subset\crit^+(\EE)$ is \emph{prime} when $\gamma$ is
not iterated, i.e., when $\gamma \neq \zeta^m$ for any integer
$m\geq2$ and $\zeta\in\crit^+(\EE)$.  For each $c\in\sigma(I)$, we
denote by $n_c(I)$ the number of non-degenerate, prime, critical
circles in $\PP^c(I)$. Theorem~\ref{mt:bars_lower_bound} has the
following corollary.

\begin{maincor}\label{c:non_degenerate}
  Let $(M,g)$ be a closed Riemannian manifold, let $I\subset SM$ be a
  locally maximal, expansive, compact invariant subset of the geodesic
  flow, and let $\F$ be any coefficient field. There exists $\delta>0$
  such that every $c\in\sigma(I)$ is a boundary point of at least
  $2n_c(I)$ bars of size at least $\delta$ in the closed geodesics
  barcode $\B(g;\F)$.
\end{maincor}

The proof of Theorem~\ref{mt:bars_lower_bound} is based on a uniform
crossing energy bound for locally maximal, expansive, compact
invariant subsets of the geodesic flow. This is
Proposition~\ref{p:crossing_energy_period_1} which can be viewed as a
Morse theoretic analogue for closed geodesics of the crossing energy
theorem from \cite{Ginzburg:2014aa, Ginzburg:2018aa}
for Hamiltonian diffeomorphisms. In the context of generating
functions of Hamiltonian diffeomorphisms, an analogous statement was
proved by Allais, \cite[Sect.~7]{Allais:2020aa}. Similarly to the
Hamiltonian version that lies at the heart of the arguments in
\cite{Cineli:2021aa, Ginzburg:2014aa, Ginzburg:2018aa}, we expect
Proposition~\ref{p:crossing_energy_period_1} to have further
applications to the study of closed geodesics.

To illustrate the crossing energy theorem for geodesic flows, consider
the particular case where the locally maximal, expansive compact
invariant set is a periodic orbit $\gamma$, which corresponds to a
critical circle $S^1\cdot\gamma\subset\crit^+(\EE)$.  The local
maximality and expansiveness imply that there exists an open
neighborhood $U\subset M$ of the support of $\gamma$ such that no
other closed geodesic has support entirely contained in $U$. This
implies that there exists $\delta_U(\gamma)>0$ with the following
property. Consider an energy gradient flow line starting close to
$\gamma$, i.e., a solution $u\colon\R\to\Lambda U$ of the ordinary
differential equation $\dot u=\nabla\EE(u)$ such that $u(0)$ is
sufficiently $C^0$-close to $\gamma$, where $\Lambda U$ the loop space
of $U$.  If $u(0)\neq\gamma$, there exists a large positive or
negative time $s$ such that the support of the loop $u(s)\in\Lambda$
is not entirely contained in the open set $U$, and we have a uniform
energy drop $|\EE(u(0))-\EE(u(s))|\geq\delta_U(\gamma)$. If we repeat
the same argument replacing $\gamma$ with its $m$-th iterate
$\gamma^m$, an analogous energy gradient flow line would have energy
drop $\delta_U(\gamma^m)$, and \emph{a priori} $\delta_U(\gamma^m)$
may depend on the order of iteration $m$ and shrink as $m$ grows.  The
uniform crossing energy bound gives a positive lower bound for
$\delta_U(\gamma^m)$ independent of $m$. Indeed, we shall prove that
$\delta_U(\gamma^m)\to\infty$, but with a caveat: we are not able to
derive such a statement in the infinite dimensional setting $\Lambda$,
and instead we shall establish it using finite dimensional
approximations of $\Lambda$, as in Milnor's \cite{Milnor:1963aa}. This
will not prevent us from obtaining Theorem~\ref{mt:bars_lower_bound}
in the infinite dimensional setting $\Lambda$.

\subsection{Organization of the paper}
After recalling some preliminary elementary facts on the closed
geodesics barcode in Section~\ref{s:preliminaries}, we prove
Theorem~\ref{mt:bar<top_plus} in Section~\ref{s:bar<top}.  Assuming
Corollary~\ref{c:non_degenerate}, in Section~\ref{s:bar>top} we
establish Theorem~\ref{mt:bar>top_I} and
Corollary~\ref{c:surfaces}. In Section~\ref{s:crossing_energy_bound}
we introduce a finite dimensional setting for the energy action
functional and prove the uniform crossing energy bound,
Proposition~\ref{p:crossing_energy_period_1}. Finally, in
Section~\ref{s:persistence}, we prove
Theorem~\ref{mt:bars_lower_bound} and
Corollary~\ref{c:non_degenerate}.

\section{Closed geodesics barcode}
\label{s:preliminaries}

Let $(M,g)$ be a closed Riemannian manifold of dimension at least two,
and let $\EE \colon \Lambda\to[0,\infty)$ be the associated energy
functional defined by \eqref{eq:EE}. The length spectrum of the
Riemannian manifold is the set
\[\sigma=\sigma(g):=
  \big\{\sqrt{\EE(\gamma)}\ \big|\ \gamma\in\crit^+(\EE)\big\}.
\]
Since $\EE$ satisfies the Palais--Smale condition, \cite{Palais:1963aa},
$\sigma$ is closed. By Sard's theorem, $\sigma$ has measure
zero, and hence it is nowhere dense. Consider the persistence module
\begin{align*}
 H_a\ttoup^{i_{b,a}} H_b,\quad 0<a<b\leq\infty,
\end{align*}
where $H_a:=H_*(\Lambda^{<a}, \Lambda^0)$ and the maps $i_{b,a}$ are the
homomorphisms induced by the inclusion. It is well known that
\begin{align*}
 H_\epsilon=\{0\},\qquad\forall\epsilon\in(0,2\injrad(g)).
\end{align*}

The vector space $H_\infty$ is always infinite-dimensional.
If the Riemannian metric $g$ is bumpy (i.e., all closed geodesics,
including the iterated ones, are non-degenerate), $H_c$ is
finite-dimensional for any $c\in(0,\infty)$ since the critical set
$\crit(\EE)\cap\EE^{-1}(0,c^2)$ consists of finitely many
non-degenerate critical circles.  On the other hand, without the bumpy
assumption on the Riemannian metric $g$, the vector space $H_c$ may be
infinite-dimensional even for some finite $c$, but only if
$c\in\sigma$.

\begin{lem}
\label{l:H_b_finite_dimensional}
For any $b\in(0,\infty)\setminus\sigma$, the relative homology
group $H_b$ is finite-dimensional.
\end{lem}

\begin{proof}
  Since the length spectrum $\sigma$ is closed, for each
  $b\in(0,\infty)\setminus\sigma$ there exists $a\in(0,b)$ such
  that $[a,b]\cap\sigma=\varnothing$. Namely, the interval $[a,b]$
  consists of regular values of the energy $\EE$. The usual gradient
  flow deformation result from Morse theory guarantees that the
  inclusion $\Lambda^{<a}\hookrightarrow\Lambda^{<b}$ is a homotopy
  equivalence and hence $i_{b,a}:H_a\to H_b$ is an isomorphism.

  By the bumpy metric theorem from \cite{Anosov:1982aa}, for any
  $\epsilon>0$ there exists a bumpy Riemannian metric $h$ on $M$ such
  that
\begin{align}
\label{e:h_g_equivalent}
 (1+\epsilon)^{-1}\|v\|_g \leq \|v\|_h \leq (1+\epsilon)\|v\|_g,
 \qquad\forall v\in TM.
\end{align}
Let $\FF \colon \Lambda\to[0,\infty)$ be the energy functional
associated with $h$, i.e.,
\begin{align*}
 \FF(\gamma)=\int_0^1 \|\dot\gamma\|_h^2\,dt.
\end{align*}
The inequalities~\eqref{e:h_g_equivalent} imply that
$(1+\epsilon)^{-2}\EE(\gamma) \leq \FF(\gamma) \leq
(1+\epsilon)^2\EE(\gamma)$ for all $\gamma\in\Lambda$. Then, choosing
$\epsilon>0$ small enough so that $(1+\epsilon)^2a<b$ and setting
$c:=(1+\epsilon)a$, we have the inclusion of sublevel sets
\begin{align*}
  \Lambda^{<a}=\EE^{-1}[0,a^2)
  \, \subseteq \, \FF^{-1}[0,c^2)
  \, \subseteq \, \EE^{-1}[0,b^2) = \Lambda^{<b}.
\end{align*}
Therefore, the isomorphism $i_{b,a}\colon H_a\to H_b$ factors through
the finite-dimensional vector space $H_*(\FF^{-1}[0,c^2),\FF^{-1}(0))$, and we
conclude that $H_b$ is finite-dimensional.
\end{proof}

\begin{lem}
  \label{l:tame}
  For all $a, b$ with $0<a<b\leq\infty$, the image $\im(i_{b,a})$ is
  finite-dimensional.
\end{lem}

\begin{proof}
  Fix a point $c\in(a,b)\setminus\sigma$. By
  Lemma~\ref{l:H_b_finite_dimensional}, $H_c$ is finite-dimensional,
  and so must be $\im(i_{b,c})$. Since $i_{b,a}=i_{b,c}\circ i_{c,a}$,
  we have $\im(i_{b,a})\subseteq \im(i_{b,c})$ and hence
  $\im(i_{b,a})$ is finite-dimensional.
\end{proof}

In the literature, persistence modules satisfying the assertion of
Lemma~\ref{l:tame} are sometimes called \emph{q-tame},
\cite{Chazal:2016aa}. To keep this paper self-contained, in the rest
of this section we briefly present some foundational results from the
theory of abstract persistence modules relevant for us and apply them in our setting.

Given any $c\in(0,\infty]$ and homology class $h\in H_c$, recall from
Section \ref{ss:barcode} the birth and death values
\begin{align*}
  \alpha(h)&:=\inf\big\{a\leq c\ \big|\
             h\in\textrm{im}(i_{c,a})\big\}\in[0,c),\\
  \beta(h)&:=\inf\big\{b>c\ \big|\
            h\in\ker(i_{b,c})\big\}\in[c,\infty],
\end{align*}
where as usual $\inf\varnothing=\infty$. If $h\neq0$, since
$H_\epsilon=\{0\}$ for a sufficiently small $\epsilon>0$, we have
$\alpha(h)\geq\epsilon$.  For all $0<a<b\leq\infty$, we define the
vector subspace 
\begin{align*}
  I_{b,a}:= \big\{h\in H_b\ \big|\ \alpha(h)\leq a\big\}
  = \bigcap_{c\in(a,b]} \im \left( i_{b,c} \right).
\end{align*}
Lemma~\ref{l:tame} implies that $I_{b,a}$ is finite-dimensional. Its
dimension $\dim I_{b,a}$ is the number of bars containing $[a,b)$.

\begin{lem}
  \label{l:I_finite}    
  For all $a, b$ with $0<a<b\leq\infty$ and for all sufficiently small
  $\epsilon\in(0,b-a)$, we have $I_{b,a}=\im(i_{b,a+\epsilon})$.
\end{lem}

\begin{proof}
  For any $\epsilon\in(0,b-a)$, we have
  $I_{b,a}\subseteq I_{b,a+\epsilon}$. This, together with the fact
  that these vector spaces are finite-dimensional, implies that
  $I_{b,a}=I_{b,a+\epsilon}$, provided that $\epsilon$ is small
  enough. Therefore, $I_{b,a}=\im(i_{b,a+\epsilon})$ for any
  sufficiently small $\epsilon\in(0,b-a)$.
\end{proof}

Let $0<a\leq b\leq\infty$ and consider the subspace of $H_a$ given
by
\begin{align*}
 K_{b,a}:=
 \big\{h\in H_a\ \big|\ \beta(h)\leq b\big\} 
 =
 \left\{
   \begin{array}{@{}ll}
     \displaystyle \bigcap_{c>b} \ker\left(i_{c,a}\right)
     & \mbox{if }b<\infty,\vspace{12pt} \\ 
    H_a & \mbox{if }b=\infty. 
  \end{array}
 \right.
\end{align*}
The dimension $\dim K_{b,a}$ is the number of bars of the form
$[a',b')$, where $a'<a\leq b'\leq b$. Notice that, unlike $I_{b,a}$,
the vector space $K_{b,a}$ can be infinite-dimensional (even when
$b<\infty$).
Next, for $c\in(a,b]$, we define
\begin{align*}
  V_{[a,b),c} & :=\big\{ h\in H_c\ \big|\ \alpha(h)\leq a,\
                \beta(h)\leq b \big\}
=K_{b,c}\cap I_{c,a}.
\end{align*}
By Lemma~\ref{l:tame}, 
$V_{[a,b),c}$ is a finite-dimensional vector subspace of $H_c$.  The
dimension $\dim V_{[a,b),c}$ is the number of bars of the form
$[a',b')$ with $a'\leq a<c\leq b'\leq b$.

\begin{lem}
\label{l:IK_finite}
If $b<\infty$, for any sufficiently small $\epsilon\in(0,b-a)$ we have                                                               
\[
  V_{[a,b),c}=\ker(i_{b+\epsilon,c})\cap \im(i_{c,a+\epsilon}).
\]
\end{lem}

\begin{proof}
  For all $\epsilon>0$, we have
  $V_{[a,b),c} \subseteq V_{[a,b+\epsilon),c}$.  Since these vector
  spaces are finite-dimensional, $V_{[a,b),c} = V_{[a,b+\epsilon),c}$
  whenever $\epsilon>0$ is sufficiently small.  Hence
  $V_{[a,b),c}= \ker(i_{b+\epsilon,c})\cap I_{c,a}$.  Finally,
  applying Lemma~\ref{l:I_finite}, we obtain the desired equality.
\end{proof}

We define the vector subspaces
\begin{align*}
  W_{[a,b),c} & :=\big\{ h\in H_c\ \big|\ \alpha(h)< a,\
                \beta(h)\leq b \big\}=K_{b,c}\cap\im(i_{c,a})
                \textrm{ \ and }\\
  Z_{[a,b),c} & :=\big\{ h\in H_c\ \big|\ \alpha(h)\leq  a,\
                \beta(h)< b \big\}=\ker(i_{b,c})\cap I_{c,a}
\end{align*}
of $V_{[a,b),c}$. In view of Lemma~\ref{l:IK_finite}, for $\epsilon>0$
small enough we have
\begin{align*}
Z_{[a,b),c} &=\ker(i_{b,c})\cap \im(i_{c,a+\epsilon})
\end{align*}
and, if $b<\infty$, 
\begin{align*}
 W_{[a,b),c} &=\ker(i_{b+\epsilon,c})\cap\im(i_{c,a}).
\end{align*}
Recall from Section~\ref{ss:barcode} that the closed geodesics barcode
$\B=\B(g;\F)$ is the collection of all pairs $([a,b),n)$, where
$[a,b)\subset(0,\infty)$ and
\[n=\dim\left(\frac{V_{[a,b),c}}{W_{[a,b),c}+Z_{[a,b),c}}\right)<\infty,
  \qquad\forall c\in(a,b].\]

\begin{lem}\label{l:n_finite}
  The dimension $\dim(V_{[a,b),c}/(W_{[a,b),c}+Z_{[a,b),c}))$ is
  independent of the value $c\in(a,b]$.
\end{lem}

\begin{proof}
  Let $a<b<\infty$. In order to simplify the notation, let us suppress
  $[a,b)$ in the notation
  and simply write $V_c=V_{[a,b),c}$, $W_c=W_{[a,b),c}$, and
  $Z_c=Z_{[a,b),c}$ for all $c\in(a,b]$. We fix $c$ and $d$ so that
  $a<c<d\leq b$, and choose $\epsilon>0$ small enough so that the
  assertion of Lemma~\ref{l:IK_finite} holds. The homomorphism
  $i_{d,c}$ restricts to surjective homomorphisms
\begin{align*}
  V_{c}=\ker(i_{b+\epsilon,c})\cap\im(i_{c,a+\epsilon})
  &\eepi^{i_{d,c}}\ker(i_{b+\epsilon,d})\cap\im(i_{d,a+\epsilon})=V_{d},
  \textrm{ \ and  }\\
  W_{c}=\ker(i_{b+\epsilon,c})\cap\im(i_{c,a})
  &\eepi^{i_{d,c}}\ker(i_{b+\epsilon,d})\cap\im(i_{d,a})=W_{d}.
\end{align*}
Consider $h\in V_c$ such that $i_{d,c}(h)=w'+z'\in W_d+Z_d$, with
$w'\in W_d$ and $z'\in Z_d$. Since $i_{d,c}(W_c)=W_d$, there exists
$w\in W_c$ such that $i_{d,c}(w)=w'$. The vector $z:=h-w$ satisfies
$i_{d,c}(z)=z'$, and since $i_{b,c}(z)=i_{b,d}(z')=0$, we infer that
$z\in Z_c$, and $h\in W_c+Z_c$. This shows that the kernel of the
induced homomorphism
\begin{align*}
 V_c \ttoup^{i_{d,c}} \frac{V_d}{W_d+Z_d}
\end{align*}
is $W_c+Z_c$, and hence $i_{d,c}$ induces an isomorphism
\begin{align*}
 \frac{V_c}{W_c+Z_c} \ttoup^{i_{d,c}}_{\cong} \frac{V_d}{W_d+Z_d}.
\end{align*}
The case of $b=\infty$ is analogous.
\end{proof}

\begin{prop}
\label{p:normal_form}
For any $c\in(0,\infty)\setminus\sigma$ and $([a,b),n)\in\B$ such
that $c\in(a,b]$, there exists an $n$-dimensional vector subspace
$B_{[a,b),c}\subset H_c$ such that $\alpha(h)=a$ and $\beta(h)=b$ for
all $h\in B_{[a,b),c}\setminus\{0\}$. These vector subspaces can be
chosen so that $H_c$ decomposes as a direct sum
\[ H_c = \bigoplus_{[a,b)\ni c} B_{[a,b),c}. \]
\end{prop}

\begin{proof}
  Note that $H_c$ is finite-dimensional as
  $c\in(0,\infty)\setminus\sigma$. Since
  $I_{c,a}\subseteq I_{c,a'}$ and $K_{b,c}\subseteq K_{b',c}$ for all
  $a<a'$ and $b<b'$, there exist finitely many values $a_i$ and $b_j$
  with
  \[0=:a_0<a_1<\ldots <a_h<c<b_1<\ldots <b_k\leq\infty\] such that
\begin{itemize}

\item $\im(i_{c,a_1})=\{0\}$ and $I_{c,a_h}=H_c$, 

\item $\im(i_{c,a})=I_{c,a_i}\subsetneq I_{c,a_{i+1}}$ for all
  $i\in\{1,\ldots ,h-1\}$ and $a\in(a_i,a_{i+1}]$, 

\item $\ker(i_{b_1,c})=H_c$ and $K_{b_k,c}=\{0\}$,

\item $\ker(i_{b,c})=K_{b_j,c}\subsetneq K_{b_{j+1},c}$ for all
  $j\in\{1,\ldots ,k-1\}$ and $b\in(b_j,b_{j+1}]$.

\end{itemize}
Namely, all bars containing $c$ are of the form $[a_i,b_j)$ for some
$i>0$ and $j$. Let
\begin{align*}
V_{j,i} & :=V_{[a_i,b_j),c}=K_{b_j,c}\cap I_{c,a_i},
\end{align*}
and notice that 
\begin{align*}
  V_{j_1,i_1}\cap V_{j_2,i_2} = V_{j_3,i_3},
  \qquad\mbox{where }i_3=\min\{i_1,i_2\},\ j_3=\min\{j_1,j_2\}.
\end{align*}
Now, set $\YY_{0,0}=\YY_{1,0}=\YY_{0,1}=\varnothing$ and choose a
basis $\YY_{1,1}$ of $V_{1,1}$. We next proceed inductively for increasing values of the integer $m=2,\dotsc ,h+k$: assume that we have already chosen the bases $\YY_{j,i}$ for $V_{j,i}$ for all $j,i$ such that $j+i<m$; for all $j,i$ such that $j+i=m$, we choose a basis $\YY_{j,i}$ of $V_{j,i}$ by completing the set $\YY_{j-1,i}\cup\YY_{j,i-1}$. At the end of the process, we obtain bases that satisfy
\begin{align*}
  \YY_{j_1,i_1}\cap \YY_{j_2,i_2} = \YY_{j_3,i_3},
  \qquad\mbox{where }i_3=\min\{i_1,i_2\},\ j_3=\min\{j_1,j_2\}.
\end{align*}
We define the vector subspaces
\begin{align*}
  B_{j,i}:=\mathrm{span}(\YY_{j,i}\setminus(\YY_{j-1,i}\cup\YY_{j,i-1}))
  \subset V_{j,i}.
\end{align*}
Notice that $\alpha(h)=a_i$ and $\beta(h)=b_j$ for all
$h\in B_{j,i}\setminus\{0\}$. The union $\YY_{j-1,i}\cup\YY_{j,i-1}$
is a basis for the vector space
\begin{align*}
V_{j,i-1} + V_{j-1,i} = W_{[a_i,b_j),c}+Z_{[a_i,b_j),c}.
\end{align*}
Thus the quotient maps $V_{j,i}\to V_{j,i}/(V_{j,i-1} + V_{j-1,i})$
restrict to isomorphisms
\begin{align}
\label{e:iso_B_j_i}
  B_{j,i} \ttoup^{\cong} \frac{V_{j,i}}{V_{j,i-1} + V_{j-1,i}}.
\end{align}
This, together with the definition of the barcode, implies that
$n_{j,i}:=\dim B_{j,i}$ is the multiplicity of the bar $[a_i,b_j)$;
namely $([a_i,b_j),n_{j,i})\in\B$.  The basis $\YY_{k,h}$ of
$V_{k,h}=H_c$ can be decomposed as a disjoint union
\begin{align*}
  \YY_{k,h} = \bigcup_{\substack{i=1,\ldots ,h\\ j=1,\ldots ,k}}
  \YY_{j,i}\setminus(\YY_{j-1,i}\cup\YY_{j,i-1}),
\end{align*}
where $\YY_{j,i}\setminus(\YY_{j-1,i}\cup\YY_{j,i-1})$ is the chosen
basis of $B_{j,i}$. Therefore, we have the direct sum decomposition
\[
 H_c=\bigoplus_{\substack{i=1,\ldots ,h\\ j=1,\ldots ,k}} \!\!\! B_{j,i}.
\qedhere
\]
\end{proof}

Recall from Section~\ref{ss:main_results} that, for any $\epsilon>0$
and $c>0$, the set $\B_{\epsilon,c}\subset\B$ is the subcollection of
those bars of size at least $\epsilon$, and $b_{\epsilon,c}$ is its
cardinality counted with multiplicities, i.e.,
\begin{align*}
 b_{\epsilon,c} = \sum_{\substack{a\leq c\\ b-a\geq\epsilon}} \dim B_{[a,b),c}.
\end{align*}

\begin{lem}
\label{l:b_epsilon_c_finite}
For all real values $\epsilon>0$ and $c>0$, we have
$b_{\epsilon,c}<\infty$.
\end{lem}

\begin{proof}
  Fix non-spectral values
  $0=:c_0,c_1,\ldots ,c_n\in(0,\infty)\setminus \sigma$ such that
  $c_n>c$ and $c_{i}<c_{i+1}<c_{i}+\epsilon$ for all
  $i\in\{0,\dotsc ,n\}$. Notice that $H_{c_i}$ is finite-dimensional
  for each $i\in\{1,\dotsc ,n\}$ by
  Lemma~\ref{l:H_b_finite_dimensional}. Moreover, any bar $[a,b)$ with
  $a\leq c$ and $b-a\geq\epsilon$ must contain some value $c_i$. Along
  with Proposition~\ref{p:normal_form}, this implies that
\[
 b_{\epsilon,c}\leq \sum_{i=1}^n \dim H_{c_i}<\infty.
\qedhere
\]
\end{proof}

\section{Upper bound on the barcode entropy}
\label{s:bar<top}

In this section we prove Theorem \ref{mt:bar<top_plus}, and briefly
discuss a minor generalization to a relative version of barcode
entropy. The proof roughly follows the same path as the proof of
\cite[Theorem~A]{Cineli:2021aa}, based on an argument in the spirit of
integral geometry involving Lagrangian tomographs. In the setting of
geodesic flows, we need to work with specific Lagrangian tomographs.

\subsection{Tomographs}
\label{ss:tomographs}
Let $(M,g)$ be a closed Riemannian manifold. If $\dim(M)\leq 1$, the
barcode entropy $\hhbar=\hhbar(g;\F)$ vanishes, and
Theorem~\ref{mt:bar<top_plus} is trivially satisfied. Therefore, from
now on we shall assume that $\dim(M)\geq2$.  For some integer $k>0$,
consider an open neighborhood $Z\subset\R^k$ of the origin, and a
smooth map \[\psi\colon Z\times M\to M,\] which we treat as a family
of maps $\psi_z=\psi(z,\cdot)\colon M\to M$ parameterized by $z\in Z$,
satisfying the following three conditions:
\begin{enumerate}
\setlength\itemsep{2pt}

\item[\reflb{itm:T1}{(i)}] $\psi_0=\id$,

\item[\reflb{itm:T2}{(ii)}] $\psi_z\colon M\to M$ is a diffeomorphism
  for each $z\in Z$,

\item[\reflb{itm:T3}{(iii)}]  the associated map
  $\Psi\colon Z\times TM\to TM\times TM$, given by
\begin{align*}
\Psi(z,v)=(d\psi_z(x)^*v,v),\ \ \forall v\in T_{\psi_z(x)}M,
\end{align*}
has surjective differential $d\Psi(z,v)$ at all points
$(z,v)\in Z\times TM\!\setminus\!\zerosection$.
\end{enumerate}

In point (iii), we denoted by
$d\psi_z(x)^*\colon T_{\psi_z(x)}M\to T_xM$ the adjoint of
$d\psi_z(x)$ with respect to the Riemannian metric $g$, which is
defined by
\[g(d\psi_z(x)^*v,w)=g(v,d\psi_z(x)w).\] In integral geometry, maps
such as $\Psi$ are sometimes referred to as \emph{tomographs}. With a
slight abuse of terminology, we shall instead call $\psi$ a tomograph.

\begin{lem}
  For some integer $k>0$ and some open neighborhood $Z$ of the origin,
  there exists a smooth map $\psi\colon Z\times M\to M$ satisfying
  conditions $\ref{itm:T1},\ref{itm:T2},\ref{itm:T3}$.
\end{lem}

\begin{proof}
  Let $\chi\colon \R^n\to[0,1]$ be a smooth bump function such that
  $\chi|_{B^n(1/2)}\equiv 1$ and $\supp(\chi)\subset B^n(1)$, where
  $B^n(r)\subset\R^n$ denotes the open ball of radius $r$ centered at
  the origin. We define the smooth map
\begin{align*}
\sigma\colon \R^{n\times n}\times\R^n\times\R^n\to\R^n,
\quad
\sigma(A,b,x)=\sigma_{A,b}(x)=x+\chi(x)\big( Ax+b \big),
\end{align*}
where $A\in\R^{n\times n}$ is treated as an $n\times n$ matrix.
Notice that, for all $(A,b)$ in a sufficiently small neighborhood
$U\subset\R^{n\times n}\times\R^n$ of the origin, the following three
conditions hold:
\begin{itemize}
\setlength\itemsep{2pt}
\item $\sigma_{A,b}\colon\R^n\to\R^n$ is a diffeomorphism, 

\item $\|\sigma_{A,b}^{-1}(0)\|<1/2$, 

\item $\sigma_{A,b}(B^n(1))=B^n(1)$.

\end{itemize}
For each non-zero vector $v\in\R^n\setminus\{0\}$, the map
$U\to\R^n\times\R^n$ given by
\begin{align}
\label{e:Sigma_construction_tomograph}
  (A,b)\mapsto(d\sigma_{A,b}(\sigma_{A,b}^{-1}(0))^Tv,
  \sigma_{A,b}^{-1}(0))=((I+A)v,-(I+A)^{-1}b)
\end{align}
has surjective differential at the origin $0\in U$. Here, the
superscript $^T$ denotes the transpose matrix, i.e., the adjoint with
respect to the Euclidean metric.

Let $(M,g)$ be a closed Riemannian manifold of dimension $n\geq
2$. For each point $x_0\in M$, 
take a chart
$\phi_{x_0}\colon W_{x_0}\to\R^n$ that provides geodesic normal
coordinates centered at $x_0$. Namely, $W_{x_0}\subset M$ is an open
neighborhood of $x_0$, $\phi_{x_0}(x_0)=0$, and, if we denote by
$(x_1,\ldots ,x_n)=\phi_{x_0}(x)$ the coordinates associated with the
chart, the Riemannian metric $g$ can be written as
\begin{align*}
 g = \sum_{i,j=1}^n g_{ij} dx_i\otimes dx_j,
\end{align*}
and its coefficients satisfy $g_{ij}(0)=\epsilon_{ij}$ and
$dg_{ij}(0)=0$. Up to modifying this chart outside a neighborhood of
$x_0$, we can assume that the image $\phi_{x_0}(W_{x_0})$ contains the
unit ball $B^n(1)$. Consider the smooth map
$\theta_{x_0}\colon U\times M\to M$ given by
\begin{align*}
 \theta_{x_0}(A,b,x)= \theta_{x_0,A,b}(x) &:= 
 \left\{
  \begin{array}{@{}ll}
    x & \mbox{if }x\in M\setminus W_{x_0}, \vspace{2pt} \\ 
    \phi_{x_0}^{-1}\circ\sigma_{A,b}\circ\phi_{x_0}(x)
       & \mbox{if }x\in W_{x_0}. 
  \end{array} 
 \right.
\end{align*}
Notice that each $\theta_{x_0,A,b}\colon M\to M$ is a
diffeomorphism. We define the lift
\begin{align*}
 \Theta_{x_0}\colon U\times TM\to TM\times TM,
 \qquad
  \Theta_{x_0}(A,b,v)
  = (d\theta_{x_0,A,b}(\theta_{x_0,A,b}^{-1}(\pi(v)))^*v,v),
\end{align*}
where $\pi\colon TM\to M$ is the base projection of the tangent
bundle.  By the surjectivity of the differential of the above
map~\eqref{e:Sigma_construction_tomograph} at the origin, we readily
see that, for all unit tangent vectors $v\in S_{x_0}M$, the
differential $d\Theta_{x_0}(0,0,v)$ is surjective. Since the
surjectivity of the differential is an open condition, there exists an
open neighborhood $V_{x_0}\subset M$ of $x_0$ such that, for each
$x\in V_{x_0}$ and $v\in S_xV_{x_0}$, the differential
$d\Theta_{x_0}(0,0,v)$ is surjective.

By the compactness of $M$, there exist finitely many points
$x_1,\ldots ,x_h\in M$ such that $V_{x_1}\cup\ldots\cup V_{x_h}=M$. We
define the map $\psi\colon U^{\times h}\times M\to M$ by
\begin{gather*}
\psi(\underbrace{A_1,b_1,\ldots ,A_h,b_h}_z,x)
=
\psi_z(x)
=
\theta_{x_1,A_1,b_1}\circ\ldots\circ\theta_{x_h,A_h,b_h}(x).
\end{gather*}
By construction, we have $\psi_0(x)=x$, and each $\psi_z\colon M\to M$
is a diffeomorphism.  The associated map
$\Psi\colon U^{\times h}\times TM\to TM\times TM$ given by
\begin{align*}
\Psi(z,v)=\Psi_z(v)=(d\psi_z(\psi_z^{-1}(x))^*v,v),
\qquad
\forall z\in U^{\times h},\ v\in T_xM
\end{align*}
has surjective differential at every point of the form
$(0,v)\in U^{\times h}\times SM$. Using one last time the fact that
the surjectivity of the differential is an open condition, we find a
sufficiently small open neighborhood $Z\subset U^{\times h}$ of the
origin such that $d\Psi(z,v)$ is surjective for all
$(z,v)\in Z\times SM$. Finally, since $v\mapsto \Psi(z,v)$ is linear
along the fibers of $TM$, we conclude that $d\Psi(z,v)$ is surjective
for all $(z,v)\in Z\times TM\setminus\zerosection$.
\end{proof}

\subsection{Variational principles associated with a
  tomograph}\label{ss:var_principle}
Fix a constant 
\begin{align*}
\epsilon\in(0,\tfrac12\injrad(g)) > 0.
\end{align*}
Using the notation from Section \ref{ss:tomographs}, up to possibly
shrinking the open neighborhood $Z\subset\R^k$ around the origin, we
can assume that the tomograph $\psi$ satisfies 
\begin{align}
\label{e:displacement_tomograph}
  \sup\big\{d(x,\psi_z(x))\ \big|\
  (z,x)\in Z\times M\big\}<\epsilon,
\end{align}
where $d\colon M\times M\to[0,\infty)$ denotes the Riemannian
distance.

We now employ $\psi$ in connection with a twisted version of the
variational principle for closed geodesics, introduced by Grove
\cite{Grove:1973aa}. For each $z\in Z$, consider the path space
\begin{align*}
  \Omega_z :=\big\{ \gamma\in\W([0,1],M)\ \big|\
  \gamma(1)=\psi_z(\gamma(0)) \big\},
\end{align*}
and denote the energy functional over it by 
\[
  \EE_z\colon \Omega_z \to[0,\infty), \qquad \EE_z(\gamma)=\int_0^1
  \|\dot\gamma\|_g^2\,dt.
\]
Notice that $\Omega_0=\Lambda$ is the usual free loop space, and
$\EE_0=\EE$ is the usual energy functional of 1-periodic curves. The
critical set $\crit^+(\EE_z):=\crit(\EE_z)\cap \EE_z^{-1}(0,\infty)$
consists of those geodesic segments $\gamma\colon [0,1]\to M$ such
that $\gamma(1)=\psi_z(\gamma(0))$ and
$d\psi_z(\gamma(0))^*\dot\gamma(1)=\dot\gamma(0)$.  As usual, we
consider the critical values of the square root of the energy
$\sqrt{\EE_z}$.
For each subset $\UU\subseteq\Omega_z$ and real number
$c>0$, we employ the usual notation $\UU^{<c}:=\UU\cap\EE_z^{-1}[0,c^2)$.
For each interval $[a,b]\subset(0,\infty)$, we set
\[\crit(\EE_z)^{[a,b]}:=\crit(\EE_z)\cap\EE_z^{-1}[a^2,b^2].\]
Notice that, by~\eqref{e:displacement_tomograph} and since
$\epsilon<\tfrac12\injrad(g)$, the interval
$[\epsilon,2\injrad(g)-\epsilon]$ does not contain critical values
of $\sqrt{\EE_z}$.

Conditions \ref{itm:T1} and \ref{itm:T2} imply the following
statement.

\begin{lem}
\label{l:sigma_nu_maps}
There exists continuous maps $\sigma\colon\Lambda\to\Omega_z $ and
$\nu\colon \Omega_z \to\Lambda$ such that
\begin{equation}
\label{e:energy_sigma_nu}
\begin{split}
\sqrt{\EE_z(\sigma(\gamma))}&<\sqrt{\EE(\gamma)}+\epsilon,
\qquad
\forall\gamma\in\Lambda,\\
\sqrt{\EE(\nu(\gamma))}&<\sqrt{\EE_z(\gamma)}+\epsilon,
\qquad
\forall\gamma\in\Omega_z .
\end{split} 
\end{equation}
Moreover, there exists a continuous homotopy
$h_s\colon\Lambda\to\Lambda$ such that $h_0=\id$,
$h_1=\nu\circ\sigma$, and
\begin{align*}
 \sqrt{\EE(h_s(\gamma))}<\sqrt{\EE(\gamma)}+2\epsilon,
 \qquad
 \forall\gamma\in\Lambda,\ s\in[0,1].
\end{align*}
Similarly, there exists a continuous homotopy
$k_s\colon\Omega_z \to\Omega_z $ such that $k_0=\id$,
$k_1=\sigma\circ\nu$, and
\begin{align*}
 \sqrt{\EE_z(k_s(\gamma))}<\sqrt{\EE_z(\gamma)}+2\epsilon,
 \qquad
 \forall\gamma\in\Omega_z ,\ s\in[0,1].
\end{align*}
\end{lem}

\begin{proof}
  For each $x,y\in M$ with $d(x,y)<\injrad(g)$, we denote by
  $\alpha_{x,y}$ the unit-speed geodesic segment joining $x$ and $y$.
  If $\gamma_1\colon [0,\tau_1]\to M$ and
  $\gamma_2\colon [0,\tau_2]\to M$ are continuous paths such that
  $\gamma_1(\tau_1)=\gamma_2(0)$, we denote their concatenation by
\[
\gamma_1*\gamma_2\colon [0,\tau_1+\tau_2]\to M,
\qquad
\gamma_1*\gamma_2(t)
=
\left\{
  \begin{array}{@{}ll}
    \gamma_1(t), & \mbox{if }t\in[0,\tau_1], \\ 
    \gamma_1(t-\tau_1), & \mbox{if }t\in[\tau_1,\tau_1+\tau_2]. 
  \end{array}
\right.
\]
For each $\W$-path $\gamma\colon [0,\tau]\to M$, let
$\overline\gamma\colon [0,1]\to M$ be the same path reparametrized
with constant speed on the unit interval, so that
\begin{align*}
  \int_0^t \|\dot{\overline\gamma}(r)\|_g\,dr=t \int_0^\tau \|
  \dot{\gamma}(r)\|_g\,dr,\qquad\forall t\in[0,1].
\end{align*}
Notice that $\sqrt{\EE(\overline\gamma)}$ is the length of the
original $\gamma$. Moreover, if $\gamma\colon [0,1]\to M$ is already
defined on the unit interval, the H\"older inequality implies that
$\EE(\overline\gamma)\leq\EE(\gamma)$. The continuous map\footnote{The
  continuity of the map $u$, as well as the continuity of the homotopy
  $u_s$, is intuitive but perhaps not obvious. A full proof was
  provided in \cite[Theorem~2]{Anosov:1980aa} by Anosov himself!}
$u\colon\Lambda\to\Lambda$, $u(\gamma)=\overline\gamma$ is homotopic
to the identity via the continuous homotopy
$u_s\colon\Lambda\to\Lambda$ defined as follows: for each
$\gamma\in\Lambda$ and $s\in[0,1]$, the curve $\gamma_s:=u_s(\gamma)$
satisfies $\gamma_s|_{[s,1]}=\gamma|_{[s,1]}$, whereas
$\gamma_s|_{[0,s]}$ is the constant speed reparametrization of
$\gamma|_{[0,s]}$. This homotopy preserves the sublevel sets of $\EE$,
for
\begin{align*}
 \EE(\gamma_s)
 &=
 \int_0^s \|\dot\gamma_s\|^2_g\, dt + \int_s^1 \|\dot\gamma_s\|^2_g\, dt 
 =
   \frac 1s \left(\int_0^s \|
   \dot\gamma_s\|_g\, dt\right)^2 + \int_s^1 \|\dot\gamma_s\|^2_g\, dt\\
 &=
   \frac 1s \left(\int_0^s \|\dot\gamma\|_g\, dt\right)^2 +
   \int_s^1 \|\dot\gamma\|^2_g\, dt
 \leq
   \int_0^s \|\dot\gamma\|^2_g\, dt + \int_s^1 \|
   \dot\gamma\|^2_g\, dt 
 = \EE(\gamma).
\end{align*}

Let $z\in Z$, and define the maps $\sigma\colon\Lambda\to\Omega_z$ and
$\nu\colon\Omega_z \to\Lambda$ by
\begin{align*}
  \sigma(\gamma) =
  \overline{\gamma*\alpha_{\gamma(1),\psi_z(\gamma(1))}},
  \qquad
  \nu(\gamma) =
  \overline{\gamma*\alpha_{\gamma(1),\psi_z^{-1}(\gamma(1))}}.
\end{align*}
Clearly, these maps are continuous and satisfy the energy
bounds~\eqref{e:energy_sigma_nu}. Observe that for the continuous
homotopy $\zeta_s\colon M\to M$ from $\zeta_0=\id$ to $\zeta_1=\psi_z$
given by
\[
  \zeta_s(x)=\alpha_{x,\psi_z(x)}(s\,d(x,\psi_z(x))) \quad
  \textrm{where} \quad 0 \le s \le 1,
\]
$d(x,\zeta_s(x))=s\,d(x,\psi_z(x))<\epsilon$ for all $x\in M$ and
$s\in[0,1]$.  Next, employing $\zeta_s$, we define another continuous
homotopy $v_s\colon\Lambda\to\Lambda$ by
\begin{align*}
  v_s(\gamma) =
  \overline{
  \gamma*\alpha_{\gamma(1),\zeta_s(\gamma(1))}*
  \alpha_{\zeta_s(\gamma(1)),\gamma(1)}}.
\end{align*}
This homotopy satisfies the properties $v_0(\gamma)=\overline\gamma$,
$v_1(\gamma)=\nu\circ\sigma(\gamma)$ and
\begin{align*}
  \sqrt{\EE(v_s(\gamma))} = \sqrt{\EE(\overline\gamma)}  +
  2s\,d(\gamma(1),\psi_z(\gamma(1)))< \sqrt{\EE(\gamma)} + 2\epsilon.
\end{align*}
Finally, we construct our desired continuous homotopy
$h_s\colon\Lambda\to\Lambda$ by juxtaposition of the homotopies $u_s$
and $v_s$, i.e.,
\begin{align*}
h_s
:=
\left\{
  \begin{array}{@{}ll}
    u_{2s}, & \mbox{if }s\in[0,1/2], \\ 
    v_{2s-1} & \mbox{if }s\in[1/2,1]. 
  \end{array}
\right.
\end{align*}
The construction of the other continuous homotopy
$k_s\colon\Omega_z\to\Omega_z$ is analogous.
\end{proof}

As in Section~\ref{ss:main_results}, we denote by $b_{\epsilon,c}$ the
number of bars of size at least $\epsilon$ and intersecting the
interval $(0,c]$ in the closed geodesics barcode $\B=\B(g;\F)$, for
a fixed coefficient field $\F$ which we suppress in the notation.

\begin{lem}
\label{l:stability_barcode}
For each $z\in Z$ such that all the critical points in
 $\crit^+(\EE_z)$ are non-degenerate, we have
\[
\#\crit(\EE_z)^{[\epsilon,c+\epsilon]}
\geq 
b_{2\epsilon,c},
\qquad
\forall c>\epsilon.
\]
\end{lem}

\begin{proof}
  Recall from Section~\ref{ss:barcode} that $\B$ is the barcode
  associated with the persistence module
  $(H_b:=H_*(\Lambda^{<b},\Lambda^0),i_{b,a})_{b>a>0}$, where the maps
  $i_{b,a}\colon H_a\to H_b$ are the homomorphisms induced by the
  inclusion. Pick $z\in Z$ such that all the critical points in
  $\crit^+(\EE_z)$ are non-degenerate. Consider the analogous
  persistence module
  $(K_b:=H_*(\Omega_z^{<b},\Omega_z^{<\epsilon}),j_{b,a})_{b>a>\epsilon}$,
  where $j_{b,a}\colon K_a\to K_b$ are again the homomorphisms induced
  by the inclusion and let $\CC$ be the barcode associated with
  $(K_b,j_{b,a})_{b>a>\epsilon}$.

  Since we chose $\epsilon<\tfrac12\injrad(g)$, the functional
  $\sqrt{\EE}$ does not have critical values in $(0,2\epsilon]$ and
  $\sqrt{\EE_z}$ does not have critical values in
  $[\epsilon,3\epsilon]$. Therefore, the inclusions
  $\Lambda^0\hookrightarrow\Lambda^{<2\epsilon}$ and
  $\Omega_z^{<\epsilon} \hookrightarrow\Omega_z^{<3\epsilon}$ are
  homotopy equivalences and induce isomorphisms
\begin{align}
\label{e:Hc_2delta_isomorphism}
  H_c=H_*(\Lambda^{<c},\Lambda^0) &\toup^{\cong}
  H_*(\Lambda^{<c},\Lambda^{<2\epsilon}),\\
\label{e:Kc_3delta_isomorphism}
  K_c=H_*(\Omega_z^{<c},\Omega_z^{<\epsilon})
     &\toup^{\cong}
  H_*(\Omega_z^{<c},\Omega_z^{<3\epsilon}).
\end{align}
The maps $\sigma\colon\Lambda\to\Omega_z$ and
$\nu\colon\Omega_z\to\Lambda$ from Lemma~\ref{l:sigma_nu_maps} induce
homomorphisms
\[\sigma_* \colon H_c\to K_{c+\epsilon},
\qquad\nu_*\colon K_{c}\to H_*(\Lambda^{<c},\Lambda^{<2\epsilon})
\]
by the first and, respectively, the second energy bound
in~\eqref{e:energy_sigma_nu}. Composing the latter homomorphisms with
the inverse of the isomorphisms~\eqref{e:Hc_2delta_isomorphism}, we
obtain homomorphisms
\begin{align*}
 \nu_*\colon K_{c}\to H_{c+\epsilon},
\end{align*}
which we still denote by $\nu_*$ with a slight abuse of notation. The
homotopies provided by Lemma~\ref{l:sigma_nu_maps}, together with the
isomorphisms \eqref{e:Hc_2delta_isomorphism} and
\eqref{e:Kc_3delta_isomorphism} induced by the inclusion, imply that
we have commutative diagrams
\[
\begin{tikzcd}[row sep=large]
H_c
\arrow[dr,"\sigma_*"]
\arrow[rr,"i_{c+2\epsilon,c}"]
&
&
H_{c+2\epsilon}
\arrow[dr,"\sigma_*"]
\\
&
K_{c+\epsilon}
\arrow[ur,"\nu_*"]
\arrow[rr,"j_{c+3\epsilon,c+\epsilon}"]
&
&
K_{c+3\epsilon}
\end{tikzcd}
\]

In the language of persistence modules, these commutative diagrams
mean precisely that $(H_b,i_{b,a})_{b>a>0}$ and
$(K_b,j_{b,a})_{b>a>\epsilon}$ are \emph{$\epsilon$-interleaved}. This
property allows us to invoke the isometry theorem from the theory of
persistence modules \cite[Theorem 5.14]{Chazal:2016aa}, which provides
a so called \emph{$\epsilon$-matching} between the barcodes $\B$ and
$\CC$. Namely, consider each barcode as a collection of intervals
(bars) in which every interval is allowed to be repeated according to
its multiplicity. For each $\delta\geq0$, we denote by
$\B_\delta\subset\B$ and $\CC_{\delta}\subset \CC$ the respective
subcollections of bars $[a,b)$ of size $b-a\geq\delta$. The
$\epsilon$-matching is a bijection $f\colon\B'\to\CC'$ from
subcollections $\B'\subset\B$ and $\CC'\subset\CC$ such that
$\B_{2\epsilon}\subset\B'$, $\CC_{2\epsilon}\subset\CC'$, and if
$f([a,b))=[c,d)$ then $[a,b)\subset[c-\epsilon,d+\epsilon)$ and
$[c,d)\subset[a-\epsilon,b+\epsilon)$. For each $\delta\geq0$ and
$c>\epsilon$, we define $\B_{\delta,c}\subset\B_{\delta}$ and
$\CC_{\delta,c}\subset\CC_{\delta}$ as the respective subcollections
of bars $[a,b)$ such that $a\leq c$. Notice that,
with the notation recalled just before the statement of the lemma,
$b_{\delta,c}$ is the cardinality of $\B_{\delta,c}$.

Since all critical points in $\crit^+(\EE_z)$ are non-degenerate, for
each compact interval $[a,b]\subset(0,\infty)$ the set of critical
points $\crit(\EE_z)^{[a,b]}$ is finite. Let $c>0$ be a critical value
of $\sqrt{\EE_z}$ and fix $\delta>0$ so small that $\sqrt{\EE_z}$ has
no critical values in the interval $(c,c+\delta]=\varnothing$. Morse
theory implies that
\begin{align*}
 \dim H_*(\Omega_z^{<c+\delta},\Omega_z^{<c})
 =
 \#\big(\crit(\EE_z)\cap\EE_z^{-1}(c^2)\big).
\end{align*}
Consider the homomorphism $j_{c+\delta,c}\colon K_c\to
K_{c+\delta}$. Notice that $\dim \ker(j_{c+\delta,c})$ is the number
of bars in $\CC$ of the form $[a,c)$ for some $a<c$, whereas
$\dim \mathrm{coker}(j_{c+\delta,c})$ is the number of bars in $\CC$
of the form $[c,b)$ for some $b>c$. The homology exact triangle
associated with the inclusion $K_c\subset K_{c+\delta}$
\[
\begin{tikzcd}[row sep=large]
K_c
\arrow[r,"j_{c+\delta,c}"]
&
K_{c+\delta}
\arrow[d]
\\
&
H_*(\Omega_z^{<c+\delta},\Omega_z^{<c})
\arrow[ul,"\partial_*"]
\end{tikzcd}
\]
implies that 
\[H_*(\Omega_z^{<c+\delta},\Omega_z^{<c})\cong
  \ker(j_{c+\delta,c})\oplus\mathrm{coker}(j_{c+\delta,c}),\] and,
therefore, every critical point in $\crit(\EE_z)\cap \EE_z^{-1}(c^2)$
is the endpoint of exactly one bar in $\CC$.

By the conclusion of the previous paragraph, for each $c>\epsilon$ the
cardinality $\CC_{0,c}$ is bounded from above by the cardinality of
$\crit(\EE_z)^{[\epsilon,c]}$. The $\epsilon$-matching $f$ readily
implies that $f(\B_{2\epsilon,c})\subset\CC_{0,c+\epsilon}$. Thus we
conclude that
\[
 b_{2\epsilon,c}
 =
 \#\B_{2\epsilon,c}
 =
 \#f(\B_{2\epsilon,c})
 \leq
 \#\CC_{0,c+\epsilon}
 \leq
 \#\crit(\EE_z)^{[\epsilon,c+\epsilon]}.
\qedhere
\]
\end{proof}

\subsection{Lagrangian intersections}

We momentarily consider the geodesic flow on the whole tangent bundle,
$\phi_t\colon TM\to TM$, $\phi_t(\dot\gamma(0))=\dot\gamma(t)$, where
$\gamma\colon \R\to M$ is a geodesic or a constant curve, and we lift
it to an embedding
\begin{align}
\label{e:Phi}
 \Phi_t\colon TM\hookrightarrow TM\times TM,
 \qquad
 \Phi_t(v)=(v,\phi_t(v)).
\end{align}
Let as above $\Psi\colon Z\times TM\to TM\times TM$ be the associated
tomograph map, which we treat as a family of maps
$\Psi_z=\Psi(z,\cdot)\colon TM\to TM\times TM$ parametrized by
$z\in Z$. There is a one-to-one correspondence
\begin{equation}
\begin{split}
\label{e:one_to_one_TM}
 \crit(\EE_z)& \ \bijection^{1:1}\ \Psi_z(TM)\cap\Phi_1(TM)\\
 \gamma&  \ \bijection\ (\dot\gamma(0),\dot\gamma(1)).
\end{split} 
\end{equation}
A straightforward calculation shows that a critical point
$\gamma\in\crit(\EE_z)$ is non-degenerate if and only if
$(\dot\gamma(0),\dot\gamma(1))$ is a transverse intersection point of
$\Psi_z(TM)$ and $\Phi_1(TM)$. Notice that, by condition \ref{itm:T3}
and the parametric transversality theorem,
$\Psi_z(TM\!\setminus\!\zerosection)\pitchfork\Phi_1(TM)$ for almost
all $z\in Z$. Therefore, for almost all $z\in Z$, all critical points
of the energy functional $\EE_z$ with positive critical values are
non-degenerate.

Since $\phi_t(c v)=c\, \phi_{ct}(v)$ for each $c>0$, the one-to-one
correspondence~\eqref{e:one_to_one_TM} can be rewritten as the
following family of one-to-one correspondences for each $c>0$:
\begin{align*}
  \crit(\EE_z)\cap\EE_z^{-1}(c^2)
  & \ \bijection^{1:1}\ \Psi_z(SM)\cap\Phi_c(SM)\\
  \gamma
  &  \ \bijection\ \tfrac1c(\dot\gamma(0),\dot\gamma(1)).
\end{align*}

Let us equip the tangent bundle $TM$ with the Sasaki Riemannian metric
$\tilde g$ induced by the Riemannian metric $g$ on $M$. This metric,
in turn, gives rise to the product Riemannian metric on $TM\times TM$
and on all its submanifolds and, in particular, on $\Phi_t(SM)$.  We
denote by $\vol(\Phi_t(SM))$ the volume obtained by integrating the
Riemannian density. Fix an open neighborhood $B\subset\R^k$ of the
origin such that $\overline B\subset Z$. Next we need the following
``Crofton inequality'' (cf.~\cite[Lemma 5.3]{Cineli:2021aa}), which is
reminiscent of the classical Crofton formula.

\begin{lem}
\label{l:Crofton_inequality}
There exists a constant $C>0$ such that, for each compact interval
$[c_0,c_1]\subset(0,\infty)$ with non-empty interior,
\begin{align*}
 \int_{B} \# \crit(\EE_z)^{[c_0,c_1]}\,dz 
 \leq 
 C
 \int_{c_0}^{c_1} \vol(\Phi_t(SM))\,dt.
\end{align*} 
\end{lem}

\begin{proof}
  Throughout the proof we employ the formalism of densities in order
  to avoid orientability issues. If
  $S$ is a submanifold of a Riemannian manifold $M$, we denote by
  $|dV_S|$ the Riemannian density of $S$ induced by the restricted
  Riemannian metric and by $\vol(S)$ the volume of $S$ with respect to
  this density, i.e.,
\begin{align*}
 \vol(S)=\int_S |dV_S|.
\end{align*}

For each $t\in\R$, set $L_t:=\Phi_t(SM)\subset SM\times SM$.  Fix a
compact interval $[c_0,c_1]\subset(0,\infty)$ with non-empty interior
and let
\begin{align}
L:=\!\!\!\bigcup_{t\in[c_0,c_1]}\!\!\! L_t\times\{t\}.
\end{align}
This is a submanifold with boundary of $SM\times SM\times
[c_0,c_1]$. Equip $[c_0,c_1]$ with the Euclidean Riemannian metric and
$SM\times SM\times [c_0,c_1]$ with the product metric. Let
$\tau\colon L\to[c_0,c_1]$ be the projection onto the last factor,
i.e., $\tau(v,\phi_t(v),t)=t$. By the smooth coarea formula (see,
e.g., \cite[Sect.~13.4.3]{Burago:1988aa}), we have 
\begin{align}
\label{e:coarea_L}
  \vol(L) = \int_{c_0}^{c_1} \bigg(\int_{L_t}
  \frac1{\|d\tau(y)\|_{\tilde g}}\,|dV_{L_t}(y)|\bigg) dt.
\end{align}
Let $X=\tfrac{d}{dt}\big|_{t=0}\phi_t$ be the geodesic vector
field on $SM$. We recall that $\|X(v)\|_{\tilde g}=\|v\|_g=1$ for all
$v\in SM$. Therefore,
\begin{align*}
 \|d\tau(v,\phi_t(v),t)\|_{\tilde g}
 =
  \frac{d\tau(v,\phi_t(v),t)(0,X(\phi_t(v)),1)}
  {\sqrt{\|X(\phi_t(v))\|_{\tilde g}^2+1}}
 =
 \frac1{\sqrt2}.
\end{align*}
Plugging this into the coarea formula~\eqref{e:coarea_L} yields
\begin{align}
\label{e:volL_volLt}
  \vol(L) = \frac1{\sqrt2} \int_{c_0}^{c_1} \vol(L_t) dt.
\end{align}

Next, let us lift $\Psi$ to a map 
\[
\tilde\Psi\colon Z\times TM\times \R\to TM\times TM\times\R,
\quad
\tilde\Psi(z,v,t)=(\Psi_z(v),t).
\]
By condition \ref{itm:T3}, $\tilde\Psi$ is a submersion outside
$B\times\zerosection\times \R$. Hence the preimage
$\tilde\Psi^{-1}(L)\subset Z\times SM\times[c_0,c_1]$ is a manifold
with boundary of dimension $\dim B$. We set
\[P:=\tilde\Psi^{-1}(L)\cap\big( B\times TM\times\R \big) \subset
  B\times SM\times[c_0,c_1]\] and let $\pi\colon P\to B$ be the
projection onto the first factor, i.e., $\pi(z,v,t)=z$. Then
\begin{align*}
N(z):= 
\# \crit(\EE_z)^{[c_0,c_1]}
=
\#\big( \pi^{-1}(z) \big)
\end{align*}
is finite for almost all $z\in B$. Therefore,
\begin{align}
\label{e:ineq_tomograph_1}
 \int_{B} N(z)\,dz
 =
 \int_{P} |\pi^*dz|.
\end{align}
Here $dz$ is the Euclidean volume form on $B$ and $|\pi^*dz|$ is the
density associated with the pullback $\pi^*dz$.  Let us equip
$E:=B\times SM\times \R$ with the product Riemannian metric which is
Euclidean on the factors $B$ and $\R$, and $P\subset E$ with the
induced Riemannian metric.  Notice that the differential $d\pi(y)$ is
a contraction for all $y\in P$; namely, the Riemannian norm of any
tangent vector $w\in T_yP$ is larger than or equal to the Euclidean
norm of its image $d\pi(y)w$. This readily implies that
$|\pi^*dz|=f|dV_{P}|$ for some smooth function $f\colon P\to[0,1]$ and
hence
\begin{align}
\label{e:ineq_tomograph_2}
\int_{P} |\pi^*dz|
\leq
\int_{P} |dV_{P}|
=
\vol(P).
\end{align}
Since $\tilde\Psi|_{P}\colon P\to L$ is a submersion, the volume of
$P$ can be computed by means of the smooth coarea formula as
\begin{align}
\label{e:coarea_P}
  \vol(P) = \int_{L} \bigg(\int_{\tilde\Psi^{-1}(y)}
  \frac1{J(w)}\,|dV_{\tilde\Psi^{-1}(y)}(w)|\bigg) |dV_{L}(y)|,
\end{align}
where $J$ denotes the Riemannian Jacobian of $\tilde\Psi|_{P}$ normal
to the fibers, i.e.,
\begin{align*}
 J(w)=\sqrt{\det\big( d\tilde\Psi|_{P}(w)\,d\tilde\Psi|_{P}(w)^*\big) }.
\end{align*}
We set
\begin{align*}
  j(w):=\min_V \sqrt{\det\big( d\tilde\Psi(w)|_V\,
  d\tilde\Psi(w)|_V^*\big) }>0,
\end{align*}
where $V$ ranges over all vector subspaces $V\subset T_w \overline E$
of dimension $\dim L$ orthogonal to the fiber
$\tilde\Psi^{-1}(\tilde\Psi(w))$.  The function $j$ is defined on the
non-compact space $\overline E=\overline B\times
SM\times\R$. Nevertheless, since $\tilde\Psi$ is the direct sum of
$\Psi$ with the identity on $\R$, we readily see that $j(z,v,t)$ is
independent of $t\in\R$ and 
hence it has a positive lower bound
\begin{align*}
 j_0:=\min_{\overline E} j>0.
\end{align*}
For any $w\in P$, if $N_{w}\subset T_w P$ is the orthogonal
complement of $\ker (d\tilde\Psi(w))$, we have
\begin{align*}
  J(w) = \sqrt{\det\big( d\tilde\Psi(w)|_{N_{w}}\,
  d\tilde\Psi(w)|_{N_{w}}^* \big)}\geq j(w)\geq j_0.
\end{align*}
Together with the coarea formula~\eqref{e:coarea_P}, this implies that
\begin{align*}
\vol(P)
\leq 
\underbrace{
\frac1{j_0}
\max_{y\in SM\times SM} \big\{\vol(\Psi^{-1}(y))\big\}}_{C}
\vol(L).
\end{align*}
Notice that the finite quantity $C$ is independent of the value $c_1$.
Finally, from this inequality combined with \eqref{e:volL_volLt},
\eqref{e:ineq_tomograph_1} and \eqref{e:ineq_tomograph_2}, we obtain
the Crofton inequality
\[
\int_{B} N(z)\,dz
\leq 
\frac C{\sqrt2}
\int_{c_0}^{c_1}\vol(L_t)\,dt.
\qedhere
\]
\end{proof}

Having established Lemmas~\ref{l:stability_barcode},
and~\ref{l:Crofton_inequality} we are now in a position to prove
Theorem~\ref{mt:bar<top_plus}. We recall that the barcode entropy
$\hhbar=\hhbar(g;\F)$ and the volume-growth entropy
$\hhvol=\hhvol(g)$ are defined by
\begin{align*}
\hhbar:=\lim_{\epsilon\to0^+}\limsup_{c\to\infty} \frac{\log^+(b_{2\epsilon,c})}c,
\qquad
\hhvol:=\limsup_{t\to\infty} \frac{\log(V(t))}t,
\end{align*}
where $V(t)=\vol(\Phi_t(SM))$.

\vspace{10pt}

\noindent\textbf{Theorem~\ref{mt:bar<top_plus}.}
\emph{On any closed Riemannian manifold and for any coefficient field,
  we have $\hhbar\leq\hhvol$.}

\begin{proof}
For each $r>\hhvol$ there exists $k>0$ such that $V(t)< k\, r\, e^{rt}$ for
  all $t>0$ large enough. Therefore
\begin{align*}
\limsup_{t\to\infty} \frac1t \log\left(\int_0^t  V(s)\,ds \right)
<
\limsup_{t\to\infty} \frac{\log(k) + rt}t  =r.
\end{align*}
In a similar vein, if $\hhvol>0$, for any $r\in(0,\hhvol)$ there exists
$k>0$ such that $V(t)> k\,r\,e^{rt}$ for all $t>0$ large enough, and therefore
\begin{align*}
\limsup_{t\to\infty} \frac1t \log\left(\int_0^t V(s) \,ds \right)
>
\limsup_{t\to\infty} \frac{\log(k) + rt}t =r.
\end{align*}
We conclude
\begin{align}
\label{e:h_int_eps_c}
 \hhvol
 =\limsup_{t\to\infty} \frac1t \log\left(\int_0^t V(s) \,ds\right)
  =\limsup_{c\to\infty}
  \frac1c \log\left(\int_{\epsilon}^{c} V(s) \,ds\right)
\end{align}
for any $\epsilon>0$. Let us fix
$\epsilon\in(0,\tfrac12\injrad(g))$, and consider a tomograph
satisfying the displacement bound~\eqref{e:displacement_tomograph}.
Lemma~\ref{l:Crofton_inequality} and \eqref{e:h_int_eps_c} imply that
\begin{align}
\label{e:h_geq_crit}
  \hhvol\geq\limsup_{c\to\infty}
  \frac1c \log^+\left(\int_{B} N_c(z)\,dz\right),
\end{align}
where $N_c(z):=\# \crit(\EE_z)^{[\epsilon,c]}$.  For almost every
$z\in B$, all critical points in $\crit^+(\EE_z)$ are
non-degenerate. By Lemma~\ref{l:stability_barcode}, for all
$c>\epsilon$ we have
\begin{align*}
  \int_{B} N_c(z)\,dz\geq \vol(B)\,b_{2\epsilon,c-\epsilon}.
\end{align*}
This inequality, together with the lower bound~\eqref{e:h_geq_crit},
implies that
\[
  \hhvol\geq\limsup_{c\to\infty}
  \frac{\log^+(\vol(B)\,b_{2\epsilon,c-\epsilon})}c 
  =
  \limsup_{c\to\infty} \frac{\log^+(b_{2\epsilon,c})}c.
\]
Finally, taking the limit for $\epsilon\to0^+$, we conclude that
$\hhvol\geq\hhbar$.
\end{proof}

\begin{rem}[Floer-theoretic approach]
  \label{rmk:Floer}
  The inequality $\hhbar\leq\hhtop$ can be proved by means of
  Floer-theoretic methods, although the argument is quite
  indirect. Here we briefly outline it.

Consider a closed Riemannian manifold $(M,g)$. With a common abuse of
notation, we still denote by $g$ the Riemannian metric on covectors,
and we consider the function $r:T^*M\to[0,\infty)$,
$r(p):=\|p\|_g$. As in Section~\ref{s:preliminaries}, we denote by
$\sigma(g)$ the length spectrum of our Riemannian manifold. For
$\epsilon>0$ small enough and $a\not\in\sigma(g)$, let
$G=G_{a,\epsilon}:T^*M\to\R$ be the piecewise smooth Hamiltonian such
that $G\equiv0$ on $\{r\leq 1+\epsilon\}$, and $G=ar-a(1+\epsilon)$ on
$\{1+\epsilon\leq r\}$. Let $F=F_{a,\epsilon}:T^*M\to\R$ be a smooth,
fiberwise monotone increasing and convex Hamiltonian approximating
$G$, and equal to $G$ outside the shell
$S=S_\epsilon=\{1\leq r\leq 1+\epsilon\}$. Notice that all the
periodic orbits of the Hamiltonian flow $\phi_{F}^t$ with integer
period lie in the compact set $\{r\leq1+\epsilon\}$.  Hamiltonians of
this type are routinely employed in the construction of the symplectic
homology. The topological entropy of the restricted Hamiltonian flow
$\phi^t_{F}|_{S}$ coincides with the topological entropy of the
reparametrized geodesic flow $\psi^t(v)=\phi^{at}(v)$, which is
$a\,\hhtop(g)$.

For $k\in\N$ and $c>0$, consider the filtered Floer complex
$CF^{(-\infty, c)}(kF)$, which is generated by the $k$-periodic orbits
of the Hamiltonian flow $\phi_{F}^t$ with action in $(-\infty,
c)$. The associated filtered Floer homology $HF^{(-\infty, c)}(kF)$
forms a persistence module, which defines the barcode entropy
$\hbar(F)$, even though the Hamiltonian $F$ is not compactly supported
and does not exactly fit into the framework of
\cite{Cineli:2021aa}. After suitable modifications, the proof of \cite[Thm.\ A]{Cineli:2021aa}
carries over, and therefore $\hhbar(F)$ is bounded from above by the
topological entropy of $\phi^t_{F}|_{S}$. In other words,
\[\hhbar(F)\leq a\,\hhtop(g).\]

With the notation from Section \ref{ss:barcode}, let
$H_a:=H_*(\Lambda^{<a},\Lambda^0)$ be the homology in the persistence
module associated with the energy functional.  The results from
\cite{Abbondandolo:2006aa, Salamon:2006aa} and \cite{Weber:2006aa} or alternatively directly from \cite{Viterbo:1999aa} assert that
$H_a\cong HF^{(-\infty, b)}(F_{a,\epsilon})$ provided $|a-b|$ and
$\epsilon>0$ are small enough. As a consequence, after a suitable manipulation of limits, it is not hard
to see that
\[\hhbar(g)=\lim_{\epsilon\to 0^+}\lim_{a\to 1^+}\hhbar(F_{a,\epsilon}),\] 
and hence
$\hhbar(g)\leq\hhtop(g)$.

After our work has appeared, the inequality $\hhbar\leq\hhtop$, for a suitable notion of barcode entropy defined via symplectic homology, has been proved by Fender, Lee, and Sohn \cite{Fender:2023aa} for the Reeb flows on the boundary of certain Liouville domains.
\end{rem}

\subsection{Relative barcode entropy}
\label{sec:relative}
The notion of barcode entropy has a relative analogue similar to the
one in \cite[Def.~2.1]{Cineli:2021aa}. Namely, let $(M,g)$ be a closed
Riemannian manifold and let $Q\subset M\times M$ be a closed
submanifold of the product. Equip $M\times M$ with the product
Riemannian metric $\tfrac12 g\oplus g$. We denote by $TQ^\perp$ the
normal bundle of $Q$ with respect of this metric and by
$SQ^\perp:=TQ^\perp\cap S(M\times M)$ the unit normal bundle of $Q$.
Let $\Omega$ be the space of $W^{1,2}$-paths $\gamma\colon [0,1]\to M$
such that $(\gamma(0),\gamma(1))\in Q$.  The energy functional
$\EE\colon \Omega\to [0,\infty)$ is again defined by the expression
\eqref{eq:EE}. The critical points of $\EE$ are geodesics or constant
curves $\gamma\in\Omega$ such that
$(-\dot\gamma(0),\dot\gamma(1))\in TQ^\perp$. Consider the geodesic
flow $\phi_t:SM\to SM$ and its lift
\begin{align*}
 \Phi_t:SM\to SM\times SM, \qquad\Phi_t(v)=(-v,\phi_t(v)).
\end{align*}
This map is the same as the one defined in~\eqref{e:Phi}, except for
the minus sign in the first factor; clearly, the exponential growth
rate of the volume of $\Phi_t(SM)$, measured with respect to the
Riemannian volume form, is still the volume-growth entropy
$\hhvol=\hhvol(g)$. The filtration $\EE^{-1}[0,a^2)\subset\Omega$
forms a persistence module, which allows us to define the
\emph{relative barcode entropy} $\hhbar(Q)=\hhbar(Q;g)$, arguing as in
Definition~\ref{def:hhbar}.  For each $c>0$, we have a one-to-one
correspondence
\begin{align*}
  \crit(\EE_z)\cap\EE_z^{-1}(c^2)
  & \ \bijection^{1:1}\ SQ^\perp\cap\Phi_c(SM)\\
  \gamma
  &  \ \bijection\ \tfrac1c(\dot\gamma(0),\dot\gamma(1)).
\end{align*}
Generalizing the arguments in the proof of
Theorem~\ref{mt:bar<top_plus}, one can prove that
\begin{align}
\label{e:relative_Thm_A}
 \hhbar(Q)\leq\hhvol.
\end{align}
In the special case where $Q=\Delta=\{(x,x)\ |\ x\in M\}$, we have
$\Omega=\Lambda$, and the relative barcode entropy $\hhbar(\Delta)$
reduces to the ordinary barcode entropy $\hhbar$. Therefore, the
inequality of Theorem~\ref{mt:bar<top_plus} is a special case
of~\eqref{e:relative_Thm_A}. On the other hand, it is not clear to us how to
extend Theorem~\ref{mt:bar>top_I} to relative barcode
entropy.

Another interesting particular case of \eqref{e:relative_Thm_A} is
when $Q=N_0\times N_1$ where $N_0$ and $N_1$ are closed submanifolds
of $M$. Then $\Omega$ is the space of paths in $M$ connecting $N_0$ to
$N_1$ and the critical points of $\EE$ are the geodesics connecting
these submanifolds and orthogonal to both of them. This version of
relative barcode entropy is a literal Morse-theoretic counterpart for
geodesic flows of \cite[Def.~2.1]{Cineli:2021aa} and
\eqref{e:relative_Thm_A} turns then into an analogue of
\cite[Thm.~5.1]{Cineli:2021aa}.

\section{Lower bound on the barcode entropy}
\label{s:bar>top}

Postponing the proof of Theorem~\ref{mt:bars_lower_bound} and
Corollary~\ref{c:non_degenerate} to the end of
Section~\ref{s:persistence}, let us now derive
Theorem~\ref{mt:bar>top_I} and Corollaries \ref{c:surfaces} from
Corollary~\ref{c:non_degenerate}.  The arguments are similar to the
ones from \cite[Thm.~B and C]{Cineli:2021aa}, and require some
familiarity with ergodic theory. Let $\phi_t\colon SM\to SM$ be the
geodesic flow on the unit tangent bundle $SM$ of a closed Riemannian
manifold $(M,g)$. Given a $\phi_t$-invariant probability measure $\mu$
on $SM$, we denote by $h_\mu$ the measure-theoretic entropy of the
geodesic flow $\phi_t$ with respect to $\mu$; see, e.g.,
\cite[Def.~4.1.1]{Fisher:2019vz}. We briefly refer to $h_\mu$ as to
the entropy of $\mu$.

\vspace{12pt}

\noindent\textbf{Theorem~\ref{mt:bar>top_I}.}
\emph{On any closed Riemannian manifold and for any coefficient field, if $I$ is a hyperbolic compact invariant subset of the geodesic flow, then $\hhbar\geq\hhtop(I)$.}

\begin{proof}
  Let $I\subset SM$ be a hyperbolic compact invariant subset of the
  geodesic flow $\phi_t\colon SM\to SM$ with positive topological
  entropy $\hhtop(I)>0$. By the variational principle for entropy,
  \cite[Cor.~4.3.9]{Fisher:2019vz}, for each $\rho>0$ there exists an
  invariant probability measure $\mu$ supported in $I$ with entropy
  $h_{\mu}\geq\hhtop(I)-\rho$.  Since $h_{\mu}$ is the average of
  $\nu\mapsto h_\nu$ over the ergodic components of $\mu$,
  \cite[Th.~8.4]{Walters:1982aa}, there exists an ergodic invariant
  probability measure $\nu$ whose support is contained in the support
  of $\mu$ such that $h_{\nu}\geq h_{\mu}-\rho$.  Since $I$ is a
  hyperbolic compact invariant subset, the restricted geodesic flow
  $\phi_t|_I$ has only one zero Lyapunov exponent at every point: the
  one corresponding to the geodesic vector field.  In particular,
  $\nu$ has only one zero Lyapunov exponent. Suppose that $\rho>0$ is
  so small that $h_{\nu}>0$.  The measure $\nu$ satisfies the
  assumptions of a theorem of Lian and Young
  \cite[Th.~D']{Lian:2012aa}, which extends to flows a result of Katok
  and Mendoza for surface diffeomorphisms
  \cite[Th.~S.5.9(1)]{Katok:1995aa}, and implies the existence of a
  locally maximal hyperbolic compact invariant subset
  $I_\rho\subset SM$ such that
  $\hhtop(I_\rho)\geq h_{\nu}-\rho\geq\hhtop(I)-2\rho$.

  Let $p(c)$ be the number of closed orbits of minimal period at most $c$ of
  the restricted geodesic flow $\phi_t|_{I_\rho}$. We recall that a
  critical circle $S^1\cdot\gamma\subset\crit^+(\EE)$ is called prime
  when the closed geodesic $\gamma$ is not iterated. By definition,
  $p(c)$ is equal to the number of prime critical circles of $\EE$ in
  $\PP^{\leq c}(I_\rho):=\PP(I_\rho)\cap\EE^{-1}(0,c^2]$. Since the
  compact invariant set $I_\rho$ is hyperbolic, we have
  \begin{align}
    \label{eq:growth_rate}
  \hhtop(I_\rho)=\limsup_{c\to\infty} \frac{\log^+(p(c))}{c},
\end{align}
see, e.g., \cite[Thm.\ 4.2.24 and Rem.\ 4.2.25]{Fisher:2019vz}.

Let $\B=\B(g;\F)$ be the closed geodesics barcode. As in
Section~\ref{ss:main_results}, we denote by $b_{\epsilon,c}$ the
number of bars in $\B$ having size at least $\epsilon$ and
intersecting the interval $(0,c]$. Since $I_\rho$ is hyperbolic, it is
expansive (see, e.g., \cite[Thm.\ 5.4.22]{Fisher:2019vz}) and all
closed geodesics tangent to $I_\rho$ are non-degenerate. Therefore,
Corollary~\ref{c:non_degenerate} provides a constant $\delta>0$ such
that $b_{\delta,c}\geq\tfrac12 p(c)$ for all $c>0$. Since
$b_{\epsilon,c}\geq b_{\delta,c}$ for all $\epsilon\in(0,\delta)$, we
have
\begin{align*}
 \hhbar
 &=  
 \lim_{\epsilon\to0^+}
 \limsup_{c\to\infty} \frac{\log^+(b_{\epsilon,c})}{c}
 \geq 
 \limsup_{c\to\infty} \frac{\log^+(p(c))}{c}
 =
 \hhtop(I_\rho)\geq
 \hhtop(I)-2\rho. 
\end{align*}
This inequality holds for an arbitrarily small $\rho>0$, and hence
$\hhbar\geq \hhtop(I)$.
\end{proof}

\vspace{12pt}

\noindent\textbf{Corollary~\ref{c:surfaces}.}
\emph{On any closed Riemannian surface and for any coefficient field,
  we have $\hhbar=\hhvol=\hhtop$.}

\begin{proof}
Theorem~\ref{mt:bar<top_plus} and Yomdin theorem provide the inequalities 
\begin{align}
\label{e:A+Yomdin}
\hhbar\leq\hhvol\leq\hhtop .
\end{align}
Assume now that our closed Riemannian manifold $M$ has dimension 2,
and its geodesic flow has positive topological entropy
$\hhtop>0$. Applying the variational principle for entropy and the
ergodic decomposition as in the proof of Theorem~\ref{mt:bar>top_I},
for each sufficiently small $\epsilon>0$ we obtain an ergodic
probability measure $\mu$ on $SM$ that is invariant under the geodesic
flow and has entropy bounded from below as
$h_{\mu}\geq \hhtop-\epsilon>0$. By the Ruelle inequality \cite[Thm.\
S.2.13]{Katok:1995aa}, the measure $\mu$ has one positive Lyapunov
exponent and one negative Lyapunov exponent. Since $SM$ has
dimension~3, $\mu$ has three Lyapunov exponents, and therefore only
one zero Lyapunov exponent: the one corresponding to the geodesic
vector field.  This allows us to apply \cite[Thm.\ D$'$]{Lian:2012aa}
as in the proof of Theorem~\ref{mt:bar>top_I}, and obtain a locally
maximal hyperbolic compact invariant subset $I_\epsilon\subset SM$
such that
$\hhtop(I_\epsilon)\geq h_{\mu}-\epsilon\geq \hhtop-2\epsilon$. This,
combined with Theorem~\ref{mt:bar>top_I}, implies that
$\hhbar \geq \hhtop-2\epsilon$ for any arbitrarily small
$\epsilon>0$, and therefore $\hhbar \geq \hhtop$. This, together
with~\eqref{e:A+Yomdin}, implies the desired identity
$\hhbar=\hhvol=\hhtop$.
  \end{proof}

\section{Crossing energy bound}\label{s:crossing_energy_bound}

The rest of the paper is devoted to the proof of 
Theorem~\ref{mt:bars_lower_bound}, which provides a uniform lower
bound for the size of the bars associated with a locally maximal,
expansive, compact invariant subset of the geodesic flow. As 
we explained 
in Section~\ref{ss:invariant_subsets}, the proof of
Theorem~\ref{mt:bars_lower_bound} hinges on a uniform crossing energy
bound, which we will establish at the end of this section in
Proposition~\ref{p:crossing_energy_period_1}, after necessary
preliminaries.

\subsection{Functional setting in period $\tau$}

Let $(M,g)$ be a closed Riemannian manifold. In order to study its
closed geodesics, it suffices to focus on the 1-periodic
ones. Nevertheless, to establish a uniform crossing energy bound, it
will be useful for us to work with closed geodesics of any period. Let
us introduce the setting.

Let $\Pi:=\Wloc(\R,M)$ be the free path space, endowed with the $\Wloc$-topology. For each $\tau>0$, we denote the subspace of
$\tau$-periodic loops by
\begin{align*}
  \Lambda_\tau =\big\{ \gamma\in\Pi\ \big|\
  \gamma(t)=\gamma(t+\tau)\ \ \forall t\in\R \big\}.
\end{align*}
Then we have the energy functional
\begin{align*}
 \EE_\tau\colon\Lambda_\tau \to[0,\infty),
 \qquad
 \EE_\tau(\gamma)=\int_0^\tau \|\dot\gamma(t)\|_g^2\,dt,
\end{align*}
whose positive critical set
$\crit^+(\EE_\tau):=\crit(\EE_\tau)\cap\EE_\tau^{-1}(0,\infty)$
consists of the $\tau$-periodic closed geodesics. For $\tau=1$, this
functional setting reduces to the one from Section~\ref{ss:barcode}.

\subsection{Finite-dimensional approximations of the path space}
Recall that the Riemannian distance of $(M,g)$ is
\begin{align*}
 d\colon M\times M\to[0,\infty),
 \qquad
 d(x,y)=\inf_\gamma \int_0^1 \|\dot\gamma(t)\|_g\, dt,
\end{align*}
where the infimum is taken over all absolutely continuous curves
$\gamma\colon [0,1]\to M$ with endpoints $\gamma(0)=x$ and
$\gamma(1)=y$.  Let
\begin{align*}
 \rho:=\injrad(g)>0
\end{align*}
be the injectivity radius of $(M,g)$. Consider the space of sequences
\begin{align*}
  P = \big\{\xx=(x_i)_{i\in\Z}\in M^{\Z}\ \big|\
  d(x_i,x_{i+1})\leq\rho/2\ \ \forall i\in\Z  \big\}. 
\end{align*}
For every $\sigma>0$, we can view $P$ as a subspace of the free
path space $\Pi$ via the embedding
\begin{align}
\label{e:PM_into_PiM}
  \iota_\sigma\colon P \hookrightarrow\Pi,
  \qquad \xx\mapsto\iota_\sigma(\xx)=\gamma_{\xx,\sigma}.
\end{align}
Here $\gamma_{\xx,\sigma}\in\Pi$ is the unique curve such that, for
each $i\in\Z$, the restriction
$\gamma_{\xx,\sigma}|_{[i\sigma,(i+1)\sigma]}$ is a smooth geodesic
segment and $\gamma_{\xx,\sigma}(i\sigma)=x_i$. Notice that the
smaller the parameter $\sigma$, the better $P$ ``approximates'' the
path space $\Pi$.

For each integer $k\geq2$, we define the space of $k$-periodic
sequences
\begin{align*}
L_k = \big\{\xx\in P \ \big|\ x_i=x_{i+k}\ \ \forall i\in\Z  \big\}.
\end{align*}
This is a finite-dimensional compact manifold with corners, of dimension 
\[\dim L_k=k\dim M.\]
Then we can identify $L_k$ with a submanifold of the free loop space
$\Lambda_{k\sigma}$ via the restriction
$\iota_\sigma\colon L_k\hookrightarrow \Lambda_{k\sigma} $ of
the 
embedding~\eqref{e:PM_into_PiM}. We denote the restriction of the energy
functional to $L_k$ by
\begin{align*}
E_{k,\sigma}:=\EE_{k\sigma}\circ\iota_{\sigma},
\qquad
E_{k,\sigma}(\xx)=\frac1\sigma\sum_{i=0}^{k-1} d(x_i,x_{i+1})^2.
\end{align*}
The positive critical set
$\crit^+(E_{k,\sigma}):=\crit(E_{k,\sigma})\cap
E_{k,\sigma}^{-1}(0,\infty)$ consists of $\xx\in L_k$ such that the
associated curve $\gamma_{\xx,\sigma}=\iota_\sigma(\xx)$ is a
$k\sigma$-periodic closed geodesic.

Let us equip the space $L_k$ with the Riemannian metric 
\[g_k(\vv,\ww):=\sum_{i=0}^{k-1}g(v_i,w_i),\] which induces the
Riemannian distance $d_k\colon L_k\times L_k\to[0,\infty)$. Notice
that $d_k$ and $d$ are related by
\begin{align*}
d_k(\xx,\yy)^2 \geq \sum_{i=0}^{k-1} d(x_i,y_i)^2 ,
\end{align*}
and actually the equality holds at least when $\xx$ and $\yy$ are
contained in the interior of $L_k$ and are sufficiently close
therein.  One may also introduce a Riemannian distance on the
infinite-dimensional free loop space $\Lambda$. Instead, we will only
need to consider the $C^0$ distance on the free path space $\Pi$,
which is given by
\begin{align*}
 d_{C^0}\colon\Pi\times\Pi\to[0,\infty),
 \qquad
 d_{C^0}(\gamma_1,\gamma_2)=\sup_{t\in\R} d(\gamma_1(t),\gamma_2(t)).
\end{align*}
When restricted to $L_k$ via the embedding $\iota_{k,\sigma}$, the
distances $d_{C^0}$ and $d_k$ are equivalent. Here we just prove the
following weaker statement, which is sufficient for our purposes.

\begin{lem}
\label{l:comparing_dist}
For every $R>0$, there exists $r>0$ with the following property. For each
$\sigma>0$ and 
$\xx,\yy\in P$ such that
$d_{C^0}(\gamma_{\xx,\sigma},\gamma_{\yy,\sigma})\geq R$, there exists
$i\in\Z$ such that $d(x_i,y_i)\geq r$.
\end{lem}

\begin{proof}
  It is sufficient to prove the lemma for the space of pairs
  \[K:=\big\{ \xx=(x_0,x_1)\in M\times M\ \big|\ d(x_0,x_1)\leq\rho/2
    \big\}.  \] For each $\xx\in K$, let
  $\gamma_{\xx}\colon [0,1]\to M$ be the unique shortest geodesic
  segment joining $x_0=\gamma_{\xx}(0)$ and $x_1=\gamma_{\xx}(1)$,
  i.e., $ \gamma_{\xx}(t)=\exp_{x_0}(t\exp_{x_0}^{-1}(x_1))$.  We denote
  the $C^0$-distance function on $K$ by
\begin{align*}
 D\colon K\times K\to[0,\infty),
 \qquad
 D(\xx,\yy)=\max_{t\in[0,1]} d(\gamma_{\xx}(t),\gamma_{\yy}(t)).
\end{align*}
The continuous function $F\colon [0,\infty)\to[0,\infty)$ defined by
\begin{align*}
  F(r)=\max\big\{D(\xx,\yy)\
  \big|\ d(x_0,y_0)^2+d(x_1,y_1)^2\leq 4r^2  \big\}
\end{align*}
is everywhere finite due to the
compactness of $K$.  Notice that $F(0)=0$, $F(r)>0$ for each $r>0$,
and $F$ is monotone increasing (perhaps not strictly monotone
increasing). Given $ R>0$, we set
\begin{align*}
 r:=\inf\big\{ r'>0\ \big|\ F(r')\geq R \big\}\in(0,\infty],
\end{align*}
where as usual $\inf\varnothing=\infty$. If $D(\xx,\yy)\geq R$ for
$\xx,\yy\in K$, then $d(x_0,y_0)^2+d(x_1,y_1)^2\geq 4r^2$, and the
inequality $d(x_i,y_i)\geq r$ must hold for at least one value of
$i\in\{0,1\}$.
\end{proof}

\subsection{Abstract crossing energy bound} 
The gradient of the energy $E_{k,\sigma}$ with respect to the
Riemannian metric $g_k$ is given by $\nabla E_{k,\sigma}(\xx) = \ww$,
where
\begin{align*}
  w_i=2\big(\dot\gamma_{\xx,\sigma}(\sigma i^-) - \dot\gamma_{\xx,\sigma}(\sigma i^
  +)\big),\qquad\forall i\in\Z.
\end{align*}
The following statement is a refinement of the classical
Palais--Smale compactness condition for the energy functional
$E_{k,\sigma}$.

\begin{lem}\label{l:compactness}
For all sequences 
\[
k_n\in\N\cap[2,\infty),
\qquad
\xx_{n}=(x_{n,i})_{i\in\Z}\in L_{k_n},
\qquad
\sigma_n\in[\sigma_0,\sigma_1]\subset(0,\infty),
\] 
satisfying 
\begin{align}
\label{e:infinitesimal_gradient}
 \lim_{n\to\infty}\|\nabla E_{k_n,\sigma_n}(\xx_n)\|_{g_{k_n}}=0,
\end{align} 
the corresponding sequence of curves
$\gamma_n:=\gamma_{\xx_n,\sigma_n}$ is compact in the $\Wloc$
topology, and any limit point of $\gamma_{n}$ is either a geodesic
$\gamma\colon\R\looparrowright M$ or a constant curve
$\gamma\equiv\gamma(0)$.
\end{lem}

\begin{proof}
  Since $M$ is compact, after extracting subsequences by means of a
  usual diagonal procedure, we can ensure that
\begin{align*}
\lim_{n\to\infty} x_{n,i}= x_i,\ \ \forall i\in\Z,
\qquad\qquad
\lim_{n\to\infty} \sigma_n=\sigma\in[\sigma_0,\sigma_1].
\end{align*}
Since $d(x_{n,i},x_{n,i+1})\leq\rho/2$, we have
$d(x_{i},x_{i+1})\leq\rho/2$ as well and hence
\[\xx=(x_i)_{i\in\Z}\in P.\] 
We set $\gamma:=\gamma_{\xx,\sigma}$ and consider the monotone increasing
affine diffeomorphisms
\[\theta_{n,i}\colon[i\sigma,(i+1)\sigma]\to[i\sigma_n,(i+1)\sigma_n].\] 
Each segment $\gamma_{n}\circ\theta_{n,i}$ converges to
$\gamma|_{[\sigma i,(i+1)\sigma]}$ in the $C^\infty$-topology, and, in
particular, the whole curves $\gamma_{n}$ converge to $\gamma$ in the
$\W_{\mathrm{loc}}$-topology. Assumption~\eqref{e:infinitesimal_gradient}
implies that
\[
  \dot\gamma(\sigma i^-)-\dot\gamma(\sigma i^+) = \lim_{n\to\infty}
  \big(\dot\gamma_{n}(\sigma_n i^-)-\dot\gamma_{n}(\sigma_n
  i^+)\big)=0,\qquad\forall i\in\Z. \]
Thus $\gamma\colon\R\to M$
is a smooth solution of the geodesic equation
$\nabla_t\dot\gamma\equiv0$.
\end{proof}

Adopting commonly used terminology, we call \emph{anti-gradient flow
  segment} of $E_{k,\sigma}$ any smooth path
$u\colon [s_0,s_1]\to L_k$ satisfying the ordinary differential
equation $\dot u=-\nabla E_{k,\sigma}(u)$.

\begin{lem}[Abstract crossing energy bound]\label{l:crossing_energy}
Let $\UU_0\subset \UU_1\subset\Pi$ be two open subsets such that 
\begin{itemize}
\item[(i)] $\overline{\UU_1\setminus \UU_0}$ does not contain any
  geodesic or any constant curve,
\item[(ii)] there exists $R>0$ such that
  $d_{C^0}(\gamma_0,\gamma_1)\geq R$ for all $\gamma_0\in \UU_0$,
  $\gamma_1\in \Pi\setminus \UU_1$.
\end{itemize}
For each $[\sigma_0,\sigma_1]\subset(0,\infty)$, there exists
$\delta>0$ with the following property: for each
$\sigma\in[\sigma_0,\sigma_1]$, $k\in\N\cap[2,\infty)$ and an
anti-gradient flow segment $u\colon [s_0,s_1]\to L_k$ of
$E_{k,\sigma}$ that crosses the shell $\UU_1\setminus \UU_0$
$($i.e.~$\iota_\sigma(u(r_0))\in \UU_0$ and
$\iota_\sigma(u(r_1))\not \in \UU_1$ for some
$r_0,r_1\in[s_0,s_1]$$)$, we have
\[E_{k,\sigma}(u(s_0))-E_{k,\sigma}(u(s_1))\geq \delta.\]
\end{lem}

\begin{proof}
  Assumption (i), together with Lemma~\ref{l:compactness}, provides a
  constant $\lambda>0$ such that
\begin{align*}
  \|\nabla E_{k,\sigma}(\xx)\|_{g_k}\geq \lambda,
  \qquad
  \forall \sigma\in[\sigma_0,\sigma_1],\ k\in\N\cap[2,\infty),\
  \xx\in L_k\cap \UU_1\setminus\UU_0. 
\end{align*}
Given $R>0$ from assumption (ii), pick 
$r>0$ as in Lemma~\ref{l:comparing_dist}.  Let $u\colon
[s_0,s_1]\to L_k$ be an anti-gradient flow segment of
$E_{k,\sigma}$ that crosses the shell $\UU_1\setminus
\UU_0$. Let
$[r_0,r_1]\subset[s_0,s_1]$ be a subinterval such that
$\iota_\sigma(u(r_i))\in\partial\UU_0$ and
$\iota_\sigma(u(r_{1-i}))\in\partial\UU_1$ for some
$i\in\{0,1\}$, and $\iota_\sigma\circ
u|_{(r_0,r_1)}$ is contained in
$\UU_1\setminus\UU_0$. By assumption~(ii) and
Lemma~\ref{l:comparing_dist}, we have
\begin{align*}
 \int_{r_0}^{r_1} \|\dot u(s)\|_{g_k}\, ds \geq d_k(u(r_0),u(r_1))\geq r.
\end{align*}
Therefore,
\begin{align*}
E_{k,\sigma}(u(s_0))-E_{k,\sigma}(u(s_1))
&=
\int_{s_0}^{s_1} \|\nabla E_{k,\sigma}(u(s))\|_{g_k}^2\, ds\\
&\geq 
\int_{r_0}^{r_1} \|\nabla E_{k,\sigma}(u(s))\|_{g_k}\|\dot u(s)\|_{g_k}\, ds\\
&\geq \lambda r =:\delta.
\qedhere
\end{align*}
\end{proof}

In the next section, we will apply the abstract crossing energy bound
from Lemma~\ref{l:crossing_energy} to anti-gradient flow segments
crossing neighborhoods of closed orbits contained in a suitable
compact invariant set.

\subsection{Locally maximal, expansive, compact invariant sets}

For each $v\in SM$, let $\gamma_v\in\Pi$ be the corresponding geodesic
such that $\dot\gamma_v(0)=v$. Notice that $\gamma_v$ is parametrized
by arc length, i.e., $\|\dot\gamma_v\|_g\equiv1$. Let $I\subset SM$ be
an invariant subset of the geodesic flow $\phi_t$. We denote the space
of unit-speed geodesics tangent to $I$ by
$\GG(I)=\big\{\gamma_v\ |\ v\in I\big\}\subset\Pi$, the subspace of
$\tau$-periodic closed geodesics among them by
$\PP_\tau(I)=\GG(I)\cap\Lambda_\tau $, and their union by
\begin{align*}
 \PP_*(I):=\bigcup_{\tau>0}\PP_\tau(I).
\end{align*}

Furthermore, we denote the open $C^0$-ball of radius $r$ centered at
$\gamma\in\Pi$ by
\begin{align*}
\UU(\gamma,r)
:=
\big\{ \zeta\in\Pi\ \big|\  d_{C^0}(\gamma,\zeta)<r \big\}.
\end{align*}
In a similar vein, we denote the open $C^0$-neighborhood of radius $r$
of a subset $\WW\subset\Pi$ by
\begin{align*}
\UU(\WW,r):=\bigcup_{\gamma\in\WW} \UU(\gamma,r).
\end{align*}

The real line $\R$ acts on the free path space $\Pi$ as
\begin{align}
\label{e:R_action}
 t\cdot\gamma=\gamma(t+\cdot)\in\Pi,\qquad t\in\R,\ \gamma\in\Pi.
\end{align}
Notice that, if $\gamma$ is a periodic curve (i.e.,
$\gamma\in\Lambda_\tau $ for some $\tau>0$), then $\R\cdot\gamma$ is
an embedded circle in $\Pi$.

\begin{lem}
\label{l:no_geodesics}
Let $I\subset SM$ be a locally maximal, expansive, compact invariant
subset of the geodesic flow. There exists $r>0$ such that, for each
$\gamma\in\PP_*(I)$, the subset
$\UU(\R\cdot\gamma,r)\setminus\R\cdot\gamma$ does not contain
geodesics or constant curves.
\end{lem}

\begin{proof}
  Recall that every geodesic $\gamma\colon\R\looparrowright M$ has
  diameter at least $\rho$, i.e.,
\begin{align*}
 \max_{t_1,t_2\in\R} d(\gamma(t_1),\gamma(t_2))\geq\rho.
\end{align*}
Therefore, for each $x\in M$ and for each geodesic
$\gamma\colon\R\looparrowright M$, there exists $t\in\R$ such that
$d(x,\gamma(t))\geq\rho/2$. In other words, $\UU(\gamma,\rho/2)$ does
not contain constant curves.

Next, arguing by contradiction, assume that there exist two sequences
of geodesics
\[\gamma_n\in\GG(I) \quad \text{and} \quad
  \zeta_n\in\UU(\gamma_n,1/n)\setminus\R\cdot\gamma_n.\] 
  We set
$\lambda_n:=\|\dot\zeta_n(0)\|_g$ and consider the associated unit
tangent vectors
\begin{align*}
  v_n:=\dot\gamma_n(0)\in I \quad \text{and} \quad
  w_n:=\frac{\dot\zeta_n(0)}{\lambda_n}\in SM.
\end{align*}
Notice that 
\[
  \phi_t(v_n)=\dot\gamma_n(t) \quad \text{and} \quad \phi_{\lambda_n
    t}(w_n)=\frac{\dot\zeta_n(t)}{\lambda_n}.\] Since
$d_{C^0}(\zeta_n,\gamma_n)<1/n$ and both $\zeta_n$ and $\gamma_n$ are
geodesics, we readily see that $\lambda_n\to1$ as $n\to\infty$, and
\begin{align}
\label{e:convergence_zetan_gamman}
  \lim_{n\to\infty} \sup_{t\in\R}
  \tilde d(\phi_{\lambda_n t}(w_n),\phi_t(v_n)) = 0,
\end{align}
where $\tilde d\colon SM\times SM\to[0,\infty)$ is the distance on the
unit tangent bundle induced by the Riemannian metric $g$.  In
particular, if $n$ is large enough, the whole orbit
$t\mapsto\phi_t(w_n)$ lies in an isolating neighborhood $U\subset SM$
of $I$ and hence $w_n\in I$ due to the local maximality of
$I$. Henceforth, we assume $n$ to be large enough so that $w_n\in I$
and $1/n<\rho/2$.

We claim that $w_n$ does not belong to the orbit
$t\mapsto\phi_t(v_n)$. Indeed, if $w_n=\phi_{t_0}(v_n)$ for some
$t_0\in\R$, we have $\zeta_n(t)=\gamma_n(t_0+\lambda_n t)$ for all
$t\in\R$.  Since $\zeta_n\not\in\R\cdot\gamma_n$, we must have
$\lambda_n\neq1$. Therefore, since $\gamma_n$ is periodic, there
exists $t_1\in\R$ such that
$\zeta_n(t_1)=\gamma_n(t_1+\rho/2)$. However, this would imply
$d(\zeta_n(t_1),\gamma_n(t_1))=\rho/2>1/n$, contradicting the fact
that $\zeta_n\in\WW(\gamma_n,1/n)$.

The assumption that $I$ is expansive provides, in particular, a
uniform lower bound $\delta>0$ with the following property: for every
$n\in\N$, since $w_n$ and $v_n$ belong to distinct orbits of the
geodesic flow, there exists $t\in\R$ such that
\begin{align*}
 \tilde d(\phi_{\lambda_n t}(w_n),\phi_{t}(v_n))\geq\delta,
\end{align*}
contradicting~\eqref{e:convergence_zetan_gamman}.
\end{proof}

We are now 
ready to prove the first version of the uniform crossing energy bound.

\begin{prop}[Crossing energy bound]
\label{p:crossing_energy}
Let $I\subset SM$ be a locally maximal, expansive, compact invariant
subset of the geodesic flow. Then, for each $r_1>0$ sufficiently
small, $r_0\in(0,r_1)$ and $[\sigma_0,\sigma_1]\subset(0,\infty)$,
there exists $\delta>0$ such that the following holds: for each
$\sigma\in[\sigma_0,\sigma_1]$, $k\in\N\cap[2,\infty)$,
$\gamma\in\PP_*(I)$ and an anti-gradient flow segment
$u\colon [s_0,s_1]\to L_k$ of $E_{k,\sigma}$ crossing the shell
$\UU(\R\cdot\gamma,r_1)\setminus\UU(\R\cdot\gamma,r_0)$, we have
\[E_{k,\sigma}(u(s_0))-E_{k,\sigma}(u(s_1))\geq \delta.\]
\end{prop}

\begin{proof}
  By Lemma \ref{l:no_geodesics}, there exists $r_1>0$ small enough so
  that, for each $\gamma\in\PP_*(I)$, the open set
  $\UU(\R\cdot\gamma,r_1)\setminus\R\cdot\gamma$ contains neither
  geodesics nor constant curves. Fix $r_0\in(0,r_1)$ and set
  $$\UU_0:=\UU(\R\cdot\gamma,r_0) \subset
  \UU_1:=\UU(\R\cdot\gamma,r_1)
  $$
  so that $\overline{\UU_1\setminus\UU_0}$ does not contain geodesics
  or constant curves. For each $\gamma_0\in \UU_0$ and
  $\gamma_1\in \Pi\setminus\UU_1$, we have
  $d_{C^0}(\gamma_0,\gamma_1)\geq r_1-r_0>0$.  This, together with
  Lemma~\ref{l:crossing_energy}, implies the proposition.
\end{proof}

\subsection{Functional setting in period one}
\label{ss:period_1}

In the previous sections, for the sake of simplicity, we worked
with the spaces $\Lambda_\tau $ of $\tau$-periodic curves. As we
already remarked, since any reparametrization with constant speed of a
geodesic is still a geodesic, in the variational theory of closed
geodesics it is enough to employ the space $\Lambda_1$ of
$1$-periodic curves.  Clearly, $\Lambda_\tau $ is diffeomorphic to
$\Lambda_1$ via the diffeomorphism
$\psi_\tau\colon\Lambda_\tau \to\Lambda_1$,
$\psi_\tau(\gamma)=\gamma(\tau\,\cdot)$.  Under $\psi_\tau$, the
energy transforms as $\EE_1\circ\psi_\tau = \tau \EE_\tau$.  From now
on, we will be working in the setting of $\tau=1$.

In the finite-dimensional reduction, we will only need one
discretization parameter, since we will set $\sigma:=1/k$.  In order
to simplify the notation, we will suppress $1$, $\tau$, $\sigma$, and
instead write, consistently with the notation from
Section~\ref{ss:barcode},
\begin{align*}
 \Lambda:=\Lambda_1,
 \qquad
 \EE=\EE_1\colon \Lambda\to[0,\infty),
 \quad
 E_{k}=E_{k,\sigma}\colon L_k\to[0,\infty).
\end{align*}
We have the already introduced $S^1=\R/\Z$-action,
$t\cdot\gamma=\gamma(t+\cdot)$, on $\Lambda$. While this action does
not preserve the subspace $\iota_{1/k}(L_k)$, it does preserve the
critical set $\iota_{1/k}(\crit(E_k))\subset\crit(\EE)$.

For a closed geodesic $\gamma\in\crit^+(\EE)$, let us fix the
discretization parameter
\[k\geq\lceil 8\EE(\gamma)^{1/2}/\rho\rceil.\]
Then
\begin{align*}
d(\gamma(t),\gamma(t+\tfrac1k))\leq\rho/8,\qquad\forall t\in S^1.
\end{align*}
In particular, $\gamma$ belongs to the image of the embedding
$\iota_{1/k}\colon L_k \hookrightarrow\Lambda$. In what follows,
with a slight abuse of notation, we will occasionally treat
$\iota_{1/k}$ as an inclusion and simply consider $L_k$ as a
subspace of the free loop space $\Lambda$.

From now on, we will only need to work with the $C^0$-neighborhoods of
a closed geodesic $\gamma\in\crit^+(\EE)$ in the free loop space
$\Lambda$. Let us denote such a neighborhood of radius $r>0$ by
\begin{align*}
  \VV(\gamma,r):=\UU(\gamma,r)\cap\Lambda=\big
  \{\zeta\in\Lambda\ \big|\ d_{C^0}(\gamma,\zeta)<r\big\},
\end{align*}
and also set
\begin{align*}
 \VV(S^1\cdot\gamma,r) := \bigcup_{t\in S^1}\VV(t\cdot\gamma,r).
\end{align*}

Let $I\subset SM$ be an invariant subset of the geodesic flow
$\phi_t$. The space $\PP(I)$ of 1-periodic closed geodesics tangent to
$I$, which was introduced in~\eqref{e:P(I)}, is related to the space
$\PP_*(I)$ of unit speed closed geodesics tangent to $I$ by
\begin{align*}
 \PP(I):= \bigcup_{\tau>0}\psi_\tau(\PP_\tau(I)).
\end{align*}

Proposition~\ref{p:crossing_energy} readily translates as follows to
the setting in period one.

\begin{prop}[Crossing energy bound in $\Lambda$]
\label{p:crossing_energy_period_1}
Let $I\subset SM$ be a locally maximal, expansive, compact invariant
subset of the geodesic flow. For every $r_1>0$ sufficiently small and
$r_0\in(0,r_1)$, there exists $\delta>0$ such that the following
holds: for each $\gamma\in\PP(I)$ and
$k:=\left\lceil (\EE(\gamma)^{1/2}+1)8/\rho\right\rceil$, any
anti-gradient flow segment $u\colon [s_0,s_1]\to L_k$ of $E_{k}$ that
crosses the shell
$\VV(S^1\cdot\gamma,r_1)\setminus \VV(S^1\cdot\gamma,r_0)$ satisfies
\[
E_{k}(u(s_0))-E_{k}(u(s_1))\geq \sqrt{\EE(\gamma)}\,\delta.
\tag*{\qed}
\]
\end{prop}

\section{Persistence of the total local homology}
\label{s:persistence}

\subsection{Local homology of closed geodesics}
\label{ss:local_homology}

Let $(M,g)$ be a closed Riemannian manifold,
$\EE\colon \Lambda\to[0,\infty)$ be its energy functional and
$\gamma\in\crit^+(\EE)$ be a closed geodesic. The \emph{local
  homology} of the critical circle $S^1\cdot\gamma$ with coefficient
in a field $\F$, suppressed in the notation, is the relative homology
group
\[
  C_*(S^1\cdot\gamma):=H_*(\Lambda^{<c}
  \cup S^1\cdot\gamma,\Lambda^{<c}),
\]
where $c:=\sqrt{\EE(\gamma)}$.

A closed geodesic $\gamma\in\crit^+(\EE)$ is said to be isolated when
there exists a neighborhood $\UU\subset\Lambda$ of the critical circle
$S^1\cdot\gamma$ such that $\UU\cap\crit(\EE)=S^1\cdot\gamma$. It is
well known that the local homology group of any isolated closed
geodesic $\gamma$ is finitely generated, and there exist arbitrarily
small open neighborhoods $\WW\subset\Lambda$ of $S^1\cdot\gamma$,
sometimes called \emph{Gromoll--Meyer neighborhoods},
\cite{Gromoll:1969uz, Gromoll:1969vn}, such that the inclusion induces
an isomorphism
\begin{align*}
  C_*(S^1\cdot\gamma)
  \toup^{\cong} H_*(\Lambda^{<c}\cup \WW,\Lambda^{<c}).
\end{align*}
Moreover, for each $\epsilon>0$ small enough, the inclusion induces a
monomorphism
$C_*(S^1\cdot\gamma) \hookrightarrow
H_*(\Lambda^{<c+\epsilon},\Lambda^{<c})$.

\subsection{Proofs of Theorem~\ref{mt:bars_lower_bound} and
  Corollary~\ref{c:non_degenerate}}

Proposition~\ref{p:crossing_energy_period_1} has the following
consequence in terms of persistence of the total local homology of a
locally maximal, expansive, compact invariant subset for the geodesic
flow. Theorem~\ref{mt:bars_lower_bound} will be a rather direct
application of it.  In the statement, we use the notation from
Section~\ref{ss:invariant_subsets}.

\begin{lem}
\label{l:persistence}
Let $(M,g)$ be a closed Riemannian manifold and $I\subset SM$ be a
locally maximal, expansive, compact invariant subset for its geodesic
flow. Then, there exists $\delta>0$ such that, for each
$c\in\sigma(I)$, the following assertions hold:
\begin{itemize}

\item[(i)] The inclusion induces a monomorphism
\[
C_*(\PP^c(I)) \hookrightarrow H_*(\Lambda^{<c+\delta },\Lambda^{<c}).
\]

\item[(ii)] The total local homology $C_*(\PP^c(I))$ is contained in
  the image of the 
  homomorphism
\[
  H_*(\Lambda^{\leq c},\Lambda^{<c-\delta })
  \to H_*(\Lambda^{\leq c},\Lambda^{<c})
\]
induced by the inclusion.

\end{itemize}
\end{lem}

\begin{proof}
  As before, we denote by $\rho:=\injrad(g)>0$ the injectivity radius.
  By Lemma~\ref{l:no_geodesics}, we can fix $R\in(0,\rho/16)$ small
  enough so that, for each $\gamma\in\PP(I)$, the $C^0$-open
  neighborhood $\VV(S^1\cdot\gamma,R)$ does not contain closed
  geodesics other than those in $S^1\cdot\gamma$ or constant curves,
  i.e.,
\begin{align}
\label{e:only_one_critical_circle}
 \VV(S^1\cdot\gamma,R)\cap\crit(\EE)=S^1\cdot\gamma,
 \qquad
 \forall \gamma\in\PP(I).
\end{align}
Up to replacing $R$ with $R/2$, we even obtain
\begin{align*}
  \VV(S^1\cdot\gamma,R)\cap \VV(S^1\cdot\zeta,R)=\varnothing,
  \qquad\forall\gamma,\zeta\in\PP(I)
  \mbox{ with }S^1\cdot\gamma\neq S^1\cdot\zeta.
\end{align*}

Take radii $r_1, r_2$ with $0<r_1<r_2<R$, where $r_2$ is small enough
so that, according to Proposition~\ref{p:crossing_energy_period_1},
there exists a constant $\delta>0$ with the following property: for
each closed geodesic $\gamma\in\PP(I)$ and for
$$
k:=\lceil (\EE(\gamma)^{1/2}+1)8/\rho\rceil,
$$
any anti-gradient flow segment $u\colon [s_0,s_1]\to L_k$ of $E_{k}$
that crosses the shell
$\VV(S^1\cdot\gamma,r_2)\setminus \VV(S^1\cdot\gamma,r_1)$ satisfies
the condition that
\begin{align}
\label{e:crossing_Zi_energy}
  E_{k}(u(s_0))-E_{k}(u(s_1))\geq \sqrt{\EE(\gamma)}\,
\frac{2(\rho+1)}{\rho}\,\delta. 
\end{align}
Compared with the statement of
Proposition~\ref{p:crossing_energy_period_1}, here we included a
multiplicative factor ${2(\rho+1)}/{\rho}$ in front of the constant
$\delta$ for aesthetic reasons, as it will simplify
inequality~\eqref{e:crossing_Zi}. 
Recall that ``\emph{crossing the shell
  $\VV(S^1\cdot\gamma,r_2)\setminus \VV(S^1\cdot\gamma,r_1)$}'' means
that $u(s_0')\in\VV(S^1\cdot\gamma,r_1)$ and
$u(s_1')\not\in\VV(S^1\cdot\gamma,r_{2})$ for some
$s_0',s_1'\in[s_0,s_1]$.  If needed, we choose $\delta$ so that
\[0<\delta<1.\]

From now on, we fix a spectral value $c\in\sigma(I)$ and work within
the sublevel set $\Lambda^{<c+1}$. We set the discretization parameter
to be $k:=\lceil (c+1)8/\rho\rceil$.  For each $\gamma\in\PP^c(I)$,
any anti-gradient flow segment $u\colon [s_0,s_1]\to L_k^{<c+1}$ of
$E_{k}$ that crosses the shell
$\VV(S^1\cdot\gamma,r_2)\setminus \VV(S^1\cdot\gamma,r_1)$ satisfies
the inequality~\eqref{e:crossing_Zi_energy}, which can be rewritten as
\begin{align*}
  \left( \sqrt{E_{k}(u(s_0))}-\sqrt{E_{k}(u(s_1))} \right)
  \left( \sqrt{E_{k}(u(s_0))}+\sqrt{E_{k}(u(s_1))} \right)
  \geq c\,
  \frac{2(\rho+1)}{\rho}\,\delta.
\end{align*}
Since
$$
\sqrt{E_{k}(u(s_0))}+\sqrt{E_{k}(u(s_1))}<2(c+1)
$$
and $c\geq\rho$, we conclude that
\begin{align}
\label{e:crossing_Zi}
\sqrt{E_{k}(u(s_0))}-\sqrt{E_{k}(u(s_1))}\geq \delta.
\end{align}

Next, consider Gromoll--Meyer neighborhoods $\WW(S^1\cdot\gamma)$ of
each critical circle $S^1\cdot\gamma\subset\PP^c(I)$, whose defining
property was recalled at the beginning of this section.  We require
such neighborhoods to be small enough so that
\begin{align*}
\WW(S^1\cdot\gamma)\Subset\Lambda^{<c+\delta}\cap\VV(S^1\cdot\gamma,r_1).
\end{align*}
Here the notation $A\Subset B$ means, as usual, that
$\overline A\subset \interior(B)$.  We also choose another
neighborhood $\YY(S^1\cdot\gamma)\subset\WW(S^1\cdot\gamma)$ of
$S^1\cdot\gamma$ that is so small that
\begin{align}
\label{e:r_gamma}
  r_\gamma:=\inf\big\{ d_k(\xx,\yy)\ \big|\
  \xx\in L_k \cap\partial\YY(S^1\cdot\gamma),\
  \yy\in L_k \cap\partial\WW(S^1\cdot\gamma)\}>0.
\end{align}
We set
\begin{align*}
  \YY&:=\bigcup_{S^1\cdot\gamma}\YY(S^1\cdot\gamma),\\
  \WW&:=\bigcup_{S^1\cdot\gamma}\WW(S^1\cdot\gamma),\\
  \VV_i&:=\bigcup_{S^1\cdot\gamma}\VV(S^1\cdot\gamma,r_i)
         \text{ for } i=1,2,\\ 
  \VV&:=\bigcup_{S^1\cdot\gamma}\VV(S^1\cdot\gamma,R),\\
  \XX&:=\Lambda^{<c+\delta}\setminus\VV_1,
\end{align*}
where the unions range over the critical circles
$S^1\cdot\gamma\subset\PP^c(I)$. Then we have
\begin{align*}
\YY\Subset\WW\Subset\VV_1\Subset\VV_2\Subset\VV,
\qquad
\crit(\EE)\cap\VV=\PP^c(I),
\qquad
\overline{\WW}\cap\overline{\XX}=\varnothing.
\end{align*}
By excision, the inclusion induces an isomorphism
\begin{align*}
  H_*(\Lambda^{<c}\cup\WW,\Lambda^{<c})\oplus
  H_*(\Lambda^{<c}\cup\XX,\Lambda^{<c})
 \toup^{\cong}
 H_*(\Lambda^{<c}\cup\WW\cup\XX,\Lambda^{<c}).
\end{align*}
By the defining property of Gromoll--Meyer neighborhoods, the inclusion
gives rise to an isomorphism
\begin{align*}
 C_*(\PP^c(I))\toup^{\cong}H_*(\Lambda^{<c}\cup\WW,\Lambda^{<c}).
\end{align*}
All together, the inclusion induces a monomorphism
\begin{align*}
 C_*(\PP^c(I))
 \hookrightarrow
 H_*(\Lambda^{<c}\cup\WW\cup\XX,\Lambda^{<c}),
\end{align*}
which fits into the following commutative diagram, all of whose arrows
are induced by inclusions.
\begin{equation*}
\begin{tikzcd}[row sep=large]
C_*(\PP^c(I))
\arrow[r,hookrightarrow]
\arrow[dr]
&H_*(\Lambda^{<c}\cup\WW\cup\XX,\Lambda^{<c})
\arrow[d,"i"]\\
&H_*(\Lambda^{<c+\delta},\Lambda^{<c})
\end{tikzcd}
\end{equation*}
In order to prove assertion (i) of Lemma \ref{l:persistence}, it
remains to show that the homomorphism $i$ is injective. To this end, it
is enough to build a continuous homotopy
\begin{align*}
 h_s\colon\Lambda^{<c+\delta}\to\Lambda^{<c+\delta},\ s\in[0,s_0],
\end{align*}
with the following properties:
\begin{itemize}

\item[(a)] $h_0=\id$,

\item[(b)] $\EE\circ h_s\leq\EE$ for all $s\in[0,s_0]$,

\item[(c)] $h_{s_0}(\Lambda^{<c+\delta})\subset\Lambda^{<c}\cup\WW\cup\XX$.

\end{itemize}
Indeed, these three conditions readily imply that the homomorphism 
\[(h_{s_0})_*\colon H_*(\Lambda^{<c+\delta},\Lambda^{<c}) \to
  H_*(\Lambda^{<c}\cup\WW\cup\XX,\Lambda^{<c})\] is a left inverse of
$i$, i.e., $(h_{s_0})_*\circ i=\id$, and hence $i$ is injective. We
shall build the homotopy $h_s$ in a few steps.

We shall 
work with the finite-dimensional loop space $L_k$.  As we already
mentioned in Subsection~\ref{ss:period_1}, with a slight abuse of
notation, we 
may view the embedding
$\iota_{1/k}\colon L_k\hookrightarrow \Lambda$ as an inclusion, and
therefore
$L_k$ as a subspace of $\Lambda$.  For each
$\gamma\in\Lambda^{< c+\delta}$ parametrized proportionally to arc length,
we have
\begin{align*}
  d(\gamma(t_1),\gamma(t_2))\leq|t_1-t_2|(c+\delta)
  < \frac\rho8,
  \qquad\forall t_1,t_2\in\R\mbox{ with }|t_1-t_2|\leq1/k.
\end{align*}
In particular, this holds for each closed geodesic
$\gamma\in\crit(\EE)\cap \Lambda^{<c+\delta}$, and hence
\begin{align*}
  d(\zeta(t_1),\zeta(t_2))<2r+\tfrac\rho8<\tfrac\rho4,
  \quad\ \ \
  \forall \zeta\in\VV(S^1\cdot\gamma,r),\ t_1,t_2\in\R
  \mbox{ with }|t_1-t_2|\leq1/k.
\end{align*}
Since the radius $R$ that we fixed
at the beginning of the proof is smaller than $\rho/16$, this inequality
implies that the closure $\overline{L_k \cap\VV(S^1\cdot\gamma,R)}$ is
contained in the interior of $L_k$, i.e.,
\begin{align}
\label{e:VV_inside_int_PkM}
 \overline{L_k \cap\VV(S^1\cdot\gamma,R)}
 \subset
 \interior(L_k).
\end{align}

We define the continuous homotopy
$h_s\colon\Lambda^{<c+\delta}\to\Lambda^{<c+\delta}$ for $s\in[0,1]$
as follows: for each $\gamma\in\Lambda^{<c+\delta}$ and $s\in[0,1]$,
the curve $\gamma_s:=h_s(\gamma)$ is given by 
$\gamma_s|_{[s,1]}=\gamma|_{[s,1]}$, whereas $\gamma_s|_{[0,s]}$ is
a constant speed reparametrization of $\gamma|_{[0,s]}$. This
homotopy clearly satisfies condition (a). It also meets condition (b),
for
\begin{align*}
  \EE(\gamma_s)
  &=
    \int_0^s \|\dot\gamma_s\|^2_g\, dt + \int_s^1 \|
    \dot\gamma_s\|^2_g\, dt 
    =
    \frac 1s \left(\int_0^s \|\dot\gamma_s\|_g\, dt\right)^2 + \int_s^1
    \|
    \dot\gamma_s\|^2_g\, dt\\
  &=
    \frac 1s \left(\int_0^s \|\dot\gamma\|_g\, dt\right)^2 + \int_s^1
    \|
    \dot\gamma\|^2_g\, dt
    \leq
    \int_0^s \|\dot\gamma\|^2_g\, dt + \int_s^1 \|
    \dot\gamma\|^2_g\, dt 
    = \EE(\gamma).
\end{align*}
The image $h_1(\Lambda^{<c+\delta})\subset\Lambda^{<c+\delta}$
consists of curves parametrized with constant speed, and therefore
\begin{align*}
  d(\gamma(t_1),\gamma(t_2))<
  \frac{c+\delta}k\leq\frac\rho8,
  \quad
  \forall \gamma\in h_1(\Lambda^{<c+\delta}),\ t_1,t_2\in[0,1]
  \mbox{ with }|t_1-t_2|\leq
  \frac1k.
\end{align*}

We extend the homotopy
$h_s\colon\Lambda^{<c+\delta}\to\Lambda^{<c+\delta}$ 
to $s\in[1,2]$ as follows: for each $\gamma\in\Lambda^{<c+\delta}$,
$s\in[0,1]$, and $i\in\{0,\ldots ,k-1\}$, set
\[h_{1+s}(\gamma)|_{[(i+s)/k,(i+1)/k]}:=h_{1}(\gamma)|_{[(i+s)/k,(i+1)/k]},\]
whereas $h_{1+s}(\gamma)|_{[i/k,(i+s)/k]}$ is the unique shortest
geodesic segment (of length less than $\rho/8$) joining
$h_{1+s}(\gamma)(i/k)$ and $h_{1+s}(\gamma)((i+1)/k)$. Clearly, by
replacing a portion of a curve with the shortest geodesic segment we
do not increase the energy and so the homotopy $h_s$ still satisfies
condition (b) for $s\in[1,2]$.  Moreover,
\begin{align*}
  h_2(\Lambda^{<c+\delta})\subset L_k\cap\Lambda^{<c+\delta}.
\end{align*}

By~\eqref{e:VV_inside_int_PkM}, we can find a smooth function
$f\colon L_k\to[0,1]$ such that
\begin{align*}
  \supp(f)\subset\overline{L_k \cap\VV}\subset\interior(L_k),
  \qquad f|_{L_k\cap\VV_2}\equiv1.
\end{align*}
We denote by $\psi_s\colon L_k\to L_k$, $s\geq0$, the flow of the
vector field $-f\nabla E_k$. Since the latter vector field is
supported in the interior of $L_k$, the flow $\psi_s$ is defined for
all times $s\in\R$. Moreover, within $L_k\cap\VV_2$, the flow
$\psi_s$ coincides with the anti-gradient flow of $E_k$ with the same
time-parametrization. We extend the definition of $h_s$ to all
$s\geq0$ by setting
\begin{align*}
 h_{2+s}:=\psi_s\circ h_2,\qquad\forall s\geq0.
\end{align*}
Clearly, condition (b) still holds for all $s\geq0$.

It remains to find $s_0>0$ such that $h_{2+s_0}$ satisfies condition
(c). Indeed, we will find $s_0>0$ such that
\begin{align}
\label{e:s0}
 \psi_{s_0}(L_k^{<c+\delta})\subset \Lambda^{<c}\cup\WW\cup\XX.
\end{align}
The critical subset $\crit(\EE)\cap \EE^{-1}(c)\subset\Lambda$ is
compact. This, along with
\eqref{e:only_one_critical_circle}, implies that $\PP^c(I)$ consists
of finitely many critical circles. Therefore, considering the radii
$r_\gamma$ defined by \eqref{e:r_gamma}, we have
\begin{align*}
  r:=\inf_{\gamma\in\PP^c(I)} r_\gamma = \inf\big\{ d_k(\xx,\yy)\ \big|\
  \xx\in L_k \cap\partial\YY,\  \yy\in L_k \cap\partial\WW\}> 0.
\end{align*}
We set
\begin{align*}
 \lambda':=\inf_{L_k\cap\VV_2\setminus\YY} \|\nabla E_k\|_{g_k}>0.
\end{align*}
For any anti-gradient flow trajectory
$u\colon [s_1,s_2]\to L_k^{<c+1}$ of $E_k$ crossing the shell
$\WW\setminus\YY$ (i.e., such that $u(s_1')\in\YY$ and
$u(s_2')\not\in\WW$ for some $s_1',s_2'\in[s_1,s_2]$) we have
\begin{align*}
 E_k(u(s_1))-E_k(u(s_2))\geq \lambda' r,
\end{align*}
and, therefore,
\begin{equation}
\label{e:crossing_W_minus_Y}
\begin{split}
  \sqrt{E_k(u(s_1))}-\sqrt{E_k(u(s_2))}
 & \geq 
 \frac{\lambda' r}{\sqrt{E_k(u(s_1))}+\sqrt{E_k(u(s_2))}} \\
 & \geq
 \frac{\lambda' r}{2(c+1)}
 =: \sigma.
\end{split}
\end{equation}
Let
\begin{align*}
\lambda
:=
  \min\left\{ \lambda' , \inf_{L_k^{\geq c+\sigma}\cap\VV_2} \|
  \nabla E_k\|_{g_k} \right\}>0,
\end{align*}
where 
$L_k^{\geq c+\sigma}:=E_k^{-1}[(c+\sigma)^2,\infty)$. The desired 
$s_0$ satisfying~\eqref{e:s0} is given by
\begin{align*}
 s_0:=\frac{(c+\delta)^2- c^2}{\lambda^2}.
\end{align*}
Let us show that \eqref{e:s0} indeed holds for this choice of $s_0$:
\begin{itemize}
\item For each $\xx\in L_k^{<c+\delta}\setminus\VV_2$, if
  $\psi_{s}(\xx)\in\VV_1$ for some $s>0$, then the anti-gradient flow
  segment $\psi_{[0,s]}(\xx)$ crosses the shell
  $\VV_{2}\setminus\VV_1$, and~\eqref{e:crossing_Zi} implies
\begin{align*}
  \sqrt{ E_k(\psi_{s}(\xx))}\leq \sqrt{E_k(\xx)}-
  \delta<c+\delta-\delta=c.
\end{align*}
In particular,
\begin{align*}
  \psi_{s_0}(L_k^{<c+\delta}\setminus\VV_2)
  \subset\Lambda^{<c}\cup\XX.
\end{align*}

\item For each $\xx\in L_k^{<c+\delta}\cap\VV_2$, we need to consider
  different subcases. If there exists $s\in(0,s_0]$ such that
  $\psi_{s}(\xx)\not\in\VV_2$, we can apply the previous point and
  infer that \[\psi_{s_0}(\xx)\in\Lambda^{<c}\cup\XX.\] Let us now
  consider the case where
\begin{align}
\label{e:in_V2_minus_<c}
\psi_{[0,s_0]}(\xx)\subset\VV_2\setminus \Lambda^{<c}. 
\end{align}
We claim that there exists $s_1\in[0,s_0]$ such that 
\[\psi_{s_1}(\xx)\in\YY^{<c+\sigma},\]
where as usual $\YY^{<c+\sigma}:=\YY\cap\Lambda^{<c+\sigma}$.
Indeed, if such an $s_1$ did not exist, we would have
$\|\nabla E_k(\psi_s(\xx))\|_{g_k}\geq\lambda$ for all $s\in[0,s_0]$
and infer that 
\begin{align*}
 \qquad 
 \sqrt{E_k(\psi_{s_0}(\xx))}
 &=
 \sqrt{E_k(\xx) - \int_0^{s_0} \|\nabla E_k(\psi_s(\xx))\|_{g_k}^2ds}\\
 &<
 \sqrt{(c+\delta)^2 - s_0\lambda^2}
 =
 c,
\end{align*}
which contradicts~\eqref{e:in_V2_minus_<c}. We also claim that
\begin{align*}
 \psi_{[s_1,s_0]}(\xx)\subset\WW.
\end{align*}
Indeed, if we had $\psi_{s}(\xx)\not\in\WW$ for some $s\in(s_1,s_0]$,
then the anti-gradient flow segment $\psi_{[s_1,s_0]}(\xx)$ would
cross the shell $\WW\setminus\YY$. Then \eqref{e:crossing_W_minus_Y}
would imply that 
\begin{align*}
 \sqrt{E_k(\psi_{s_0}(\xx))}
 \leq
 \sqrt{E_k(\psi_{s_1}(\xx))}-\sigma
 <
 c+\sigma-\sigma
 =
 c,
\end{align*}
contradicting once again~\eqref{e:in_V2_minus_<c}.  Overall, we have
proved that
\begin{align*}
 \psi_{s_0}(\xx)\in\Lambda^{<c}\cup\WW\cup\XX,
\end{align*}
and hence \eqref{e:s0} holds.
\end{itemize}
This completes the proof of assertion (i) of Lemma
\ref{l:persistence}. The proof of assertion~(ii) is similar.
\end{proof}

\begin{proof}[Proof of Theorem~\ref{mt:bars_lower_bound}]
  Let $\delta>0$ be the constant given by
  Lemma~\ref{l:persistence}. For each $c\in\sigma(I)$,
  Lemma~\ref{l:persistence}(i) implies that the inclusion induces a
  monomorphism
\begin{align}
\label{e:persistence_above}
 C_*(\PP^c(I))\hookrightarrow H_*(\Lambda^{<c+\delta},\Lambda^{<c}).
\end{align}
In particular, we have the monomorphism
\[C_*(\PP^c(I))\hookrightarrow H_*(\Lambda^{\leq c},\Lambda^{<c}),\]
where $\Lambda^{\leq c}:=\EE^{-1}[0,c^2]$. Henceforth, we will
view the total local homology $C_*(\PP^c(I))$ as a vector subspace
of $H_*(\Lambda^{\leq c},\Lambda^{<c})$.  Consider the exact triangle
\begin{equation}
\label{e:exact_triangle}
\begin{tikzcd}[row sep=large]
H_*(\Lambda^{<c},\Lambda^0)
\arrow[r,"i"]
&H_*(\Lambda^{\leq c},\Lambda^0)
\arrow[d,"j"]\\
&H_*(\Lambda^{\leq c},\Lambda^{<c})
\arrow[ul,"\partial_*"]
\end{tikzcd}
\end{equation}
where all arrows except the connecting homomorphism $\partial_*$ are
induced by the inclusions. Let $C\subset H_*(\Lambda^{\leq c},\Lambda^0)$ be
a complement of $\im(i)$, i.e.,
$H_*(\Lambda^{\leq c},\Lambda^0)=\im(i)\oplus C$.  Notice that $j|_{C}$ is 
a monomorphism and set 
\[B:=j|_{C}^{-1}(C_*(\PP^c(I))).\] For each $b\in(c,\infty]$, consider
the homomorphism 
\begin{align*}
 j_b\colon B\to H_*(\Lambda^{<b},\Lambda^0)
\end{align*}
induced by the inclusion. If $B$ is non-trivial, there exist
$b_k>\ldots >b_1>c$, with $b_k\in(c,\infty]$, and a direct sum
decomposition $B=B_{b_k}\oplus\ldots \oplus B_{b_1}$ such that
\begin{align*}
 \ker(j_b)
 =
 \left\{
  \begin{array}{@{}ll}
    \{0\} & \mbox{if }b\in(c,b_1],\vspace{5pt} \\ 
    B_{b_i}\oplus\ldots\oplus B_{b_1}
           & \mbox{if }b\in(b_i,b_{i+1}],\ i\in\{1,...,k-1\},\vspace{5pt} \\ 
    B &  \mbox{if }b>b_k.\\ 
  \end{array}
 \right.
\end{align*}
This readily implies that $([c,b_i),n_i)\in\B$ for all
$i=1,\ldots ,k$, where $n_i\geq\dim B_{b_i}$.  Moreover, the
injectivity of~\eqref{e:persistence_above} together with the following
commutative diagram whose arrows are induced by inclusions
\[
\begin{tikzcd}[row sep=large]
B
\arrow[r,"j_{c+\delta}"]
\arrow[d,"j"',hookrightarrow]
&H_*(\Lambda^{< c+\delta},\Lambda^0)
\arrow[d]\\
C_*(\PP^c(I))
\arrow[r,hookrightarrow]
&H_*(\Lambda^{<c+\delta},\Lambda^{<c})
\end{tikzcd}
\]
implies that $b_1>c+\delta$. Therefore, all bars $[c,b_i)$ have size at
least $\delta$.

Now, consider the subgroup 
\begin{align*}
 A:=\partial_*(C_*(\PP^c(I)))\subset H_*(\Lambda^{<c},\Lambda^0)
\end{align*}
and, for each $a<b$, the homomorphisms 
\begin{align*}
 i_{b,a}\colon H_*(\Lambda^{<a},\Lambda^0)\to H_*(\Lambda^{<b},\Lambda^0)
\end{align*}
induced by the inclusion. If $A$ is non-trivial, there exist
$a_1<\ldots <a_h<c$ with $a_1>0$ and a direct sum decomposition
$A=A_{a_1}\oplus\ldots\oplus A_{a_h}$ such that
\begin{align*}
 \im(i_{c,a})\cap A
 =
 \left\{
  \begin{array}{@{}ll}
    \{0\} & \mbox{if }a\in(0,a_1],\vspace{5pt} \\ 
    A_{a_1}\oplus\ldots \oplus A_{a_i}
           &  \mbox{if }a\in(a_i,a_{i+1}],\ i\in\{1,...,h-1\},\vspace{5pt} \\ 
    A &  \mbox{if }a\in(a_h,c].\\ 
  \end{array}
 \right.
\end{align*}
The exact triangle~\eqref{e:exact_triangle} implies that
$A\subseteq\ker i_{b,c}$ for all $b>c$. Therefore,
$([a_i,c),m_i)\in\B$ for all $i=1,\ldots ,h$, where
$m_i\geq\dim A_{a_i} $.  By Lemma~\ref{l:persistence}(ii),
$C_*(\PP^c(I))$ is contained in the image of the homomorphism
\[
  H_*(\Lambda^{\leq c},\Lambda^{<c-\delta})
  \to H_*(\Lambda^{\leq c},\Lambda^{<c})
\]
induced by the inclusion. Inserting this homomorphism into the
following commutative diagram, whose horizontal rows are induced by
the inclusions,
\begin{equation*}
\begin{tikzcd}[row sep=large]
H_*(\Lambda^{\leq c},\Lambda^{<c-\delta})
\arrow[r]\arrow[d,"\partial_*"']
&H_*(\Lambda^{\leq c},\Lambda^{<c})
\arrow[d,"\partial_*"]\\
H_*(\Lambda^{<c-\delta},\Lambda^0)
\arrow[r]
&H_*(\Lambda^{<c},\Lambda^0)
\end{tikzcd}
\end{equation*}
we readily see that $c>a_h+\delta$. Therefore, all bars $[a_i,c)$ have
size at least $\delta$.  Finally, the exact
triangle~\eqref{e:exact_triangle} also implies that
$C_*(\PP^c(I)) \cong B\oplus A$, and hence
\[
  n_1 + \ldots + n_k + m_1 + \ldots + m_h
  \geq \dim B+\dim A=\dim C_*(\PP^c(I)).
\qedhere
\]
\end{proof}

\begin{proof}[Proof of Corollary \ref{c:non_degenerate}]
  Since the compact invariant subset $I\subset SM$ is locally maximal
  and expansive, 
  every closed geodesic $\gamma\in\PP(I)$ is isolated. For each
  $c\in\sigma(I)$, the total local homology of $\PP(I)$ splits as a
  direct sum
\begin{align*}
  C_*(\PP^c(I))
  \cong
  \!\!\!\bigoplus_{S^1\cdot\gamma\subset\PP^c(I)}
  \!\!\!\!\! C_*(S^1\cdot\gamma),
\end{align*}
due to the excision property of singular homology.

If a critical circle $S^1\cdot\gamma\subset\crit^+(\EE)$ is
non-degenerate and prime, its local homology with any coefficient has
dimension $\dim C_*(S^1\cdot\gamma)=2$. This is well known to the
experts (see, e.g., \cite[Prop.~3.8(i)]{Bangert:2010aa}), but we
sketch the argument for the reader's convenience. Let
$\E^-\subset T_{\gamma}\Lambda$ be a vector subspace of maximal
dimension over which the Hessian $d^2\EE(\gamma)$ is negative
definite. The dimension of $\E^-$ is equal to Morse index of $\gamma$,
which is finite. Consider an embedded ball $\NN_\gamma\subset\Lambda$
of dimension $\dim \NN_\gamma =\dim \E^-$ containing $\gamma$ in its
interior and having tangent space $T_\gamma \NN_\gamma=\E^-$. Up to
shrinking $\NN_\gamma$ around $\gamma$, we have
$\EE(\zeta)<c:=\EE(\gamma)$ for all
$\zeta\in\NN_\gamma\setminus\{\gamma\}$, and the map
$S^1\times \NN_\gamma\to\Lambda$, $(t,\gamma)\mapsto t\cdot\gamma$ is
a homeomorphism onto a neighborhood $\NN\subset\Lambda$ of the
critical circle $S^1\cdot\gamma$. The inclusion
$(\NN,\NN^{<c})\hookrightarrow(\Lambda^{<c}\cup
S^1\cdot\gamma,\Lambda^{<c})$ induces an isomorphism in homology, and
hence
\begin{align*}
 C_*(S^1\cdot\gamma)
 &\cong
 H_*(\NN,\NN^{<c})\\
 &\cong
   H_*(S^1\times \NN_\gamma,S^1\times
   \NN_\gamma\setminus\{\gamma\})\\
 &\cong
 H_{*-\dim \E^-}(S^1)\\
 &\cong
 \F\oplus\F.
\end{align*}

All together, we have proved that $\dim C_*(\PP^c(I))\geq 2 n_c(I)$,
where $n_c(I)$ denotes the number of non-degenerate, prime, critical
circles in $\PP^c(I)$. Combined with
Theorem~\ref{mt:bars_lower_bound}, this
implies 
Corollary \ref{c:non_degenerate}: there exists $\delta>0$ such that
every $c\in\sigma(I)$ is a boundary point of at least $2n_c(I)$ bars
of size at least $\delta$ in the closed geodesics barcode
$\B$.
\end{proof}

\subsection*{Acknowledgments}
The authors are grateful to Marcelo Alves, Gonzalo Contreras, Sylvain
Crovisier, Matthias Meiwes, Gabriel Paternain, and Omri Sarig for
useful discussions. Parts of this work were carried out while the
first two authors were visiting the \'Ecole normale sup\'erieure de
Lyon, France, in May 2022, funded by the ANR CoSyDy (ANR-CE40-0014),
and also during the \emph{Symplectic Dynamics Beyond Periodic Orbits}
workshop at the Lorentz Center, Leiden, the Netherlands, in August
2022. The authors thank these institutes for their warm hospitality,
support, and stimulating working environment.

\bibliography{_biblio}
\bibliographystyle{amsalpha}

\end{document}